\documentclass{amsart}

\usepackage{amsmath,amssymb,amsbsy,amsthm,amsfonts,mathtools,dsfont,mathrsfs}

\usepackage{enumitem}
\usepackage{wrapfig}
\usepackage[top=3.5cm,bottom=3.5cm,left=3.5cm,right=3.5cm]{geometry}

\usepackage{graphicx, caption}

\allowdisplaybreaks[3]

\DeclareMathOperator{\derivaceM}{d}

\newcommand{\N}{\mathbb{N}}
\newcommand{\R}{\mathbb{R}}

\newcommand{\tr}{\operatorname{Tr}}

\begin{document}


\def\note#1{\marginpar{\small #1}}
\def\tens#1{\pmb{\mathsf{#1}}}
\def\vec#1{\boldsymbol{#1}}
\def\norm#1{\left|\!\left| #1 \right|\!\right|}
\def\fnorm#1{|\!| #1 |\!|}
\def\abs#1{\left| #1 \right|}
\def\ti{\text{I}}
\def\tii{\text{I\!I}}
\def\tiii{\text{I\!I\!I}}

\newcommand{\loc}{{\rm loc}}
\def\diver{\mathop{\mathrm{div}}\nolimits}
\def\grad{\mathop{\mathrm{grad}}\nolimits}
\def\Div{\mathop{\mathrm{Div}}\nolimits}
\def\Grad{\mathop{\mathrm{Grad}}\nolimits}
\def\cof{\mathop{\mathrm{cof}}\nolimits}
\def\det{\mathop{\mathrm{det}}\nolimits}
\def\lin{\mathop{\mathrm{span}}\nolimits}
\def\pr{\noindent \textbf{Proof: }}

\def\pp#1#2{\frac{\partial #1}{\partial #2}}
\def\dd#1#2{\frac{\d #1}{\d #2}}
\def\vec#1{\boldsymbol{#1}}

\def\0{\vec{0}}
\def\A{\mathcal{A}}
\def\B{\mathcal{B}}
\def\b{\vec{b}}
\def\C{\mathcal{C}}
\def\c{\vec{c}}
\def\D{\vec{Dv}}
\def\DD{\vec{D}}
\def\e{\varepsilon}
\def\f{\vec{f}}
\def\F{\vec{F}}
\def\g{\vec{g}}
\def\G{\vec{G}}
\def\h{\vec{h}}
\def\I{\vec{I}}
\def\k{\vec{k}}
\def\n{\vec{n}}
\def\NN{\mathcal{N}}
\def\S{\vec{S}}
\def\s{\vec{s}}
\def\sg{\vec{\sigma}}
\def\T{\vec{T}}
\def\u{\vec{u}}
\def\vp{\vec{\varphi}}
\def\vv{\vec{v}}
\def\w{\vec{w}}
\def\W{\vec{W}}
\def\x{\vec{x}}
\def\z{\vec{z}}
\def\Z{\vec{Z}}

\def\Ae{\A_\e}
\def\Aee{\Ae^\e}
\def\Be{\B_\e}
\def\Bee{\Be^\e}
\def\De{\DD^\e}
\def\Dve{\D^\e}
\def\oD{\overline{\DD}}
\def\tD{\tilde{\DD}}
\def\Dn{\DD^\e}
\def\Dno{\overline{\Dn}}
\def\Dnt{\tilde{\Dn}}
\def\Dm{\DD^\eta}
\def\Dmo{\overline{\Dm}}
\def\Dmt{\tilde{\Dm}}
\def\Se{\S^\e}
\def\se{\sg^\e}
\def\ose{\overline{\se}}
\def\oS{\overline{\S}}
\def\tS{\tilde{\S}}
\def\Sn{\S^\e}
\def\Sno{\overline{\Sn}}
\def\Snt{\tilde{\Sn}}
\def\Sm{\S^\eta}
\def\Smo{\overline{\Sm}}
\def\Smt{\tilde{\Sm}}
\def\U{\mathcal{U}}
\def\ve{\vv^\e}
\def\ove{\overline{\ve}}
\def\vd{\vv^\delta}
\def\sd{\sg^\delta}
\def\Sd{\S^\delta}
\def\Dd{\D^\delta}

\def\Wnd#1{W^{1,#1}_{\n, \diver}}
\def\Wndr{W^{1,r}_{\n, \diver}}

\def\o{\Omega}
\def\po{\partial \Omega}
\def\dt{\frac{d}{dt}}
\def\pt{\partial_t}
\def\ig{\int_{\Gamma} \!}
\def\igt{\int_{\Gamma_t} \!}
\def\io{\!\int_{\Omega} \!}
\def\ipo{\!\int_{\partial \Omega} \!}
\def\iq{\int_{Q} \!}
\def\iqt{\int_{Q_t} \!}
\def\it{\int_0^t \!}
\def\iT{\int_0^T \!}
\def\d{\, \derivaceM\!}

\def\mn{\mathcal{P}}
\def\du{\mathcal{W}}
\def\tr{\text{tr}~}
\def\tow{\rightharpoonup}


\newtheorem{theorem}{Theorem}[section]
\newtheorem{lemma}[theorem]{Lemma}
\newtheorem{proposition}[theorem]{Proposition}
\newtheorem{remark}[theorem]{Remark}
\newtheorem{corollary}[theorem]{Corollary}
\newtheorem{definition}[theorem]{Definition}
\newtheorem{example}[theorem]{Example}
\newtheorem*{theorem*}{Theorem}

\numberwithin{equation}{section}

\title{On the dynamic slip boundary condition for Navier--Stokes-like problems}

\thanks{The research of A.~Abbatiello is supported by Einstein Foundation, Berlin.  A.~Abbatiello is also member of the Italian National Group for the Mathematical Physics (GNFM) of INdAM. M. Bul\'{\i}\v{c}ek acknowledges the support of the project  No. 20-11027X financed by Czech Science Foundation (GA\v{C}R). M. Bul\'{\i}\v{c}ek is member of the Jind\v{r}ich Ne\v{c}as Center for Mathematical Modelling. E.~Maringov\'a acknowledges support from Charles University Research program UNCE/SCI/023, the grant SVV-2020-260583 by the Ministry of Education, Youth and Sports, Czech Republic and from the Austrian Science Fund (FWF), grants P30000, W1245, and F65.}

\author[A.~Abbatiello]{Anna Abbatiello}
\address{Institut f\"ur Mathematik, Technische Universit\"at Berlin, Stra{\ss}e des 17. Juni 136, 10623 Berlin-Charlottenburg, Germany}
\email{anna.abbatiello@tu-berlin.de}

\author[M.~Bul\'{i}\v{c}ek]{Miroslav Bul\'i\v{c}ek}
\address{Mathematical Institute, Faculty of Mathematics and Physics, Charles University, Sokolovsk\'{a} 83, 186~75, Prague, Czech Republic}
\email{mbul8060@karlin.mff.cuni.cz}

\author[E.~Maringov\'{a}]{Erika Maringov\'{a}}
\address{Institute for Analysis and Scientific Computing, Vienna University of Technology, Wiedner Hauptstr. 8-10, 1040 Vienna, Austria}
\email{erika.maringova@tuwien.ac.at}


\begin{abstract}
The choice of the boundary conditions in mechanical problems has to reflect the interaction of the considered material with the surface, despite the assumption of the no-slip condition is preferred to avoid boundary terms in the analysis and slipping effects are usually overlooked. Besides the ``static slip models", there are phenomena not accurately described by them, e.g. in the moment when the slip changes rapidly, the wall shear stress and the slip  can exhibit a sudden overshoot and subsequent  relaxation. When these effects become significant, the so-called dynamic slip phenomenon occurs. We develop a mathematical analysis of Navier-Stokes-like problems with dynamic slip boundary condition, which requires a proper generalisation of the Gelfand triplet and the corresponding function spaces setting.
\end{abstract}

\maketitle

\section{Introduction}\label{intro}

In fluid mechanics, the flows of homogeneous incompressible fluids are driven, at the macroscopic level, by the incompressibility condition, the balance equations for the linear momentum and for the angular momentum complemented with the constitutive equations.  These laws are partial differential equations, describing the change of and the relation between the relevant quantities, namely the velocity of the fluid $\vv$, the symmetric part of the velocity gradient  $2\D:= \left(\nabla \vv + (\nabla \vv)^\top\right)$, the Cauchy stress tensor $\T$ (especially its deviatoric part $\S := \T - \frac{\tr \T}{3}\I$), the pressure $p=\frac{\tr \T}{3}$ and the given density of external body forces  $\f$. The constitutive equations in the bulk explain the material properties of the fluid and on the boundary its interaction with the surroundings.
Such system of PDEs in a bounded domain is completed prescribing the boundary and initial conditions for the crucial variables. The boundary conditions can be viewed as constitutive relations at the interface between two materials. In particular, no-slip and static slip models are not always valid according to measurements (references can be found e.g. in~\cite[Section~6.2]{H}).
Therefore, motivated by \cite{H} and the references therein, our aim in this study is to perform an analysis for the so-called dynamic slip phenomenon on the boundary of the domain. In this setting, we consider the impermeable boundary, i.e., the normal component of the velocity remains zero, while the tangential part of the velocity and of its time derivative is related to the wall shear stress $\s$ via the following formula
\begin{equation}\label{dynamic-slip}
\s = \alpha \sg + \beta \pt \vv
\end{equation}
with $\alpha, \beta>0$ and where $\sg$ represents an auxiliary stress vectorial function, typically dependent on $\vv$. Boundary condition \eqref{dynamic-slip} enables us to capture the non-monotone behaviour of the slip velocity on the boundary. To the best of our knowledge there are no analytical results for models of type \eqref{dynamic-slip}. We prove the long-time and large-data existence results for the evolutionary flows of models that follow the Navier--Stokes-like structure prescribing the dynamic slip condition on the boundary.

We want to emphasize at the very beginning that the presence of the time derivative of the velocity of the fluid in the boundary condition essentially change the setting of the problem. In classical problems of fluid mechanics with Dirichlet or slip boundary conditions the underlying function space is just a subspace of Sobolev or Lebesgue spaces and consists of functions having zero divergence, which in addition  have zero  normal component at the boundary. However, here it would not be a proper space and we would not have a proper Gelfand triplet to introduce the meaning of the time derivative on the boundary. Note that here the difficulty does not come from the convective term and one  has to face the same problem also for the Stokes flow. In addition, in the setting of the present paper, we need to prescribe the initial data also on the boundary, which must be reflected in the analysis. Therefore, we must invent a new function space setting and a new concept of (weak) solution, which satisfy two essential properties:
\begin{itemize}
\item[1)]  The concept of a weak solution is compatible with the notion of classical solution, i.e., a weak solution which is sufficiently regular is also a classical solution.
\item[2)] The concept of a weak solution is compatible with the standard notion of weak solution for Dirichlet or slip boundary conditions.
\end{itemize}
These two tasks can be viewed as a continuation of the program initiated  by Leray \cite{Leray} who developed the mathematical theory for Navier--Stokes equation in the whole $\mathbb{R}^3$ and later extended by Hopf \cite{Hopf}, who developed the concept of a weak solution also in bounded domains with Dirichlet data and established its existence. Hence, our result goes in the spirit of Leray and Hopf and provides the framework for essentially new boundary conditions. Furthermore, although it is not the goal of the paper, the theory built here allows one to introduce a proper notion of the Stokes semigroup related to the dynamic slip models and therefore we have a new concept of mild solutions for dynamic slip models, which may be a starting point for subsequent analysis of dynamic slip models from many different perspectives.

Finally, we want to point out that we do not restrict ourselves to a Navier--Stokes model with linear dynamic slip boundary conditions only, but we consider rather general class of fluids with very complicated rheology. Indeed, we look at the constitutive equations in the bulk and on the boundary in terms of maximal monotone graphs.  Particularly, the constitutive relations between $\S$ and $\D$ are expressed through a maximal monotone $r$-graph, while $\sg$ and $\vv$ are related via a maximal monotone $2$-graph.

\subsection*{Problem formulation}

The theoretical background for our problem is formulated for general dimension $d\geq 1$. However, we apply the results only in the dimension $d=3$. We consider a bounded time interval $(0,T)$ and a Lipschitz domain $\o \subset \R^3$, and denote $Q:=(0,T) \times \o$ the time--space domain and $\Gamma:= (0,T) \times \po$ its spatial boundary. We study the relations between the velocity field $\vv: Q \to \R^3$, the deviator of the Cauchy stress tensor $\S: Q\to \R^{3\times 3}$, and the pressure $p:Q\to \R$. We denote by $\n:\Gamma \to \R^3$ the outward unit normal vector to the boundary and by $\f:Q \to \R^3$ the given external forces. Also, the initial velocity $\vv_0: \overline{\o} \to \R^3$ is given. Finally, we consider parameters $\alpha, \beta \geq 0$ and $r \in (6/5,\infty)$.

The incompressibility condition, the balance of linear momentum, the boundary conditions for the velocity (impermeability of the boundary, implying that the velocity on the boundary only acts in the tangential direction, $\vv_\tau = \vv$ on $\Gamma$) and for the stress (here, we use the standard notation $\s:=-(\S\n)_{\tau}$ for the shear stress, and $\sg$ is an auxiliary function which has no physical meaning, but serves to relate the shear stress $\s$ to the slip velocity $\vv_\tau$ via \eqref{NSrg}), and the initial condition for the velocity\footnote{If $\beta=0$, then \eqref{NSrd} does not see the time derivative and the initial condition \eqref{NSre} is prescribed only in $\o$.} are
\begin{subequations}\label{problem}
\begin{align}
\diver \vv &=0 &&\text{in } Q,\label{NSrb}\\
\pt \vv + \diver(\vv \otimes \vv)- \diver \S + \nabla p &= \f &&\text{in } Q, \label{NSra}\\
\vv\cdot\n&=0 &&\text{on }  \Gamma,\label{NSrc}\\
-(\S \n)_\tau=: \s &= \alpha \sg + \beta \pt \vv &&\text{on } \Gamma,\label{NSrd}\\
\vv(0)&=\vv_0  &&\text{in } \overline{\o},\label{NSre}
\end{align}
As was already mentioned, the balance of angular momentum guarantees the symmetry of the stress tensor $\T=\T^\top$ (and therefore also $\S=\S^\top$), which will be considered and not explicitly repeated throughout the work. To complete the problem, we need to prescribe the constitutive equations relating $\S$ and $\sg$ to $\D$ and $\vv$. In general, we consider the constitutive relations
\begin{align}
(\S,\D) &\in \A &&\text{in } Q,\label{NSrf}\\
(\sg,\vv)  &\in \B &&\text{on } \Gamma,\label{NSrg}
\end{align}
\end{subequations}
where $\A$ is a maximal monotone $r$-graph and $\B$ is a maximal monotone $2$-graph (see Definition~\ref{maxmongrafA}). The value of the parameter~$r$ characterizes the response of the fluid inside the domain (for illustration, see Figure \ref{fig:responses}), on the other hand, the parameters $\alpha$ and $\beta$ determine the slip regime on the boundary (as summarized in \eqref{alfabeta}).

Regarding the boundary condition \eqref{NSrd}, we could in principle prescribe some surface force $\g: \Gamma \to \R^3$. Such an equation would look like
\begin{equation}\label{genP5}
 \alpha \sg + \beta \pt \vv =\s + \g ~\text{ on } \Gamma.
\end{equation}
It would lead to two classes of external forces -- $\f$, representing the external body forces in $Q$ (like for example the gravitational force), and $\g$, representing the surface forces on $\Gamma$. Such a  generalization is definitely possible, and we refer to \eqref{dualityVf} and the description therein for more details. However, on $\Gamma$, the external surface forces usually cause the deformation of the boundary. Since our domain, as well as its boundary, is always given and fixed, such forces are of no physical relevance and should not be considered. Therefore, we simply set $\g\equiv \0$.

\subsection*{Implicit theory - the role of parameter \texorpdfstring{$r$}{r}}

We briefly explain the use of the maximal monotone graphs in the formulation of the constitutive relations and the importance of the parameter $r$. The class of implicit models is commonly described via some function $\G$, or equivalently, via graph $\A$, defined as
\begin{equation}\label{impli}
\G(\S,\D) = \0  ~~\Longleftrightarrow~~ (\S,\D) \in \A.
\end{equation}
For physical reasons, it is natural to impose some assumptions on the function $\G$, or equivalently, on the graph $\A$. Namely, we require that the origin belongs to the graph; that the shear rate is non-decreasing with respect to the shear stress\footnote{This holds for the fluids whose microstructure does not affect their mechanical properties.}; and that the energy dissipation $\xi = \S :\D$ is not only positive, but also provides some useful information - here, $r$ enters the game. Depending on the information, we can talk about different classes of graphs (for details, see the definition of the maximal monotone graph (Definition \ref{maxmongrafA})).

\begin{figure}[!tbp]
\centering
\includegraphics[width=0.9\textwidth]{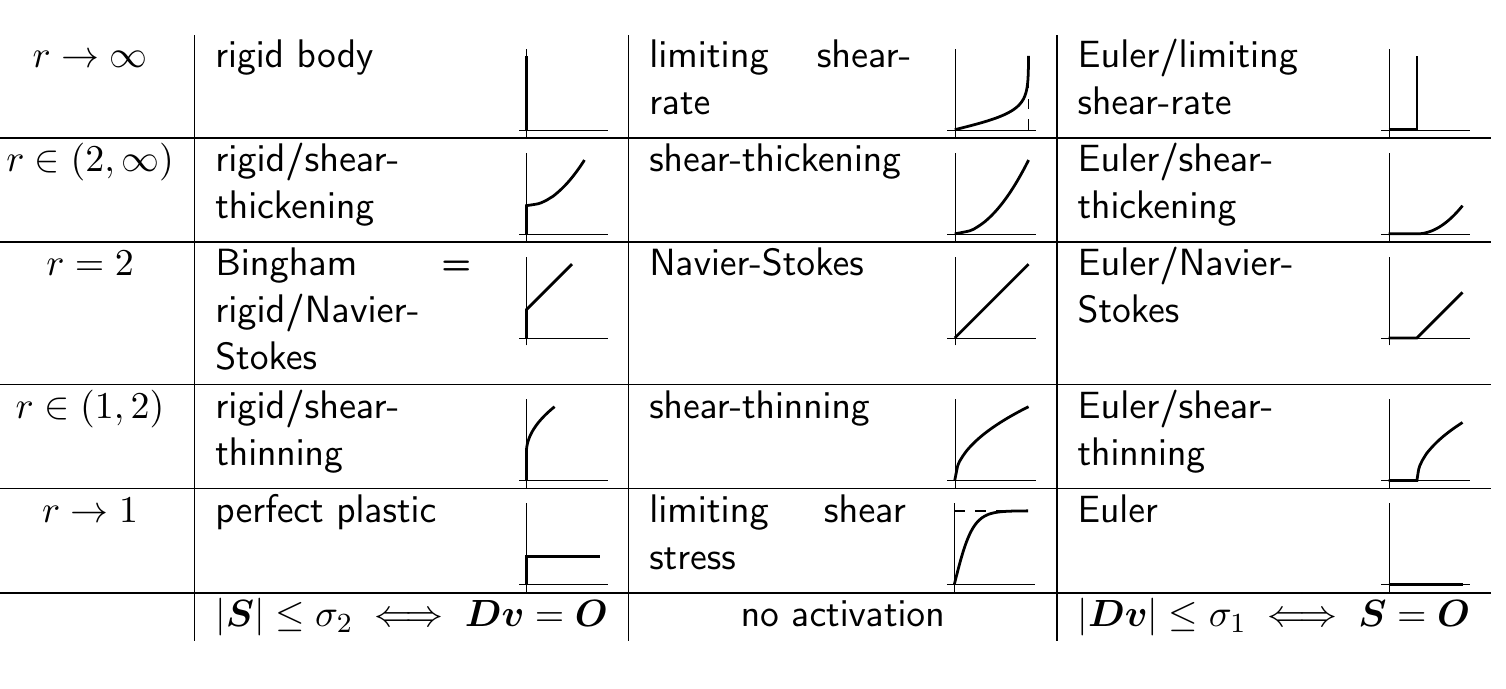}
\caption{A systematic classification of fluid-like responses with respect to the power-law index $r$ and the activating effect. The table includes corresponding $|\S|$ vs $|\D|$ diagrams, where $\sigma_2$ and $\sigma_1$, respectively, are the activating coefficients. Reproduced (and adjusted) with kind permission from  \cite[Table 2.1]{BlMaRa}}
\label{fig:responses}
\end{figure}

The credit for the study of the models of the type~\eqref{impli} is given to the works~\cite{Raj1, Raj2, RajSrin} and~\cite{BGMS1, BGMS2}. A systematic classification of such class of models is provided for example in~\cite{BlMaRa}. In~\cite{BuMaMa}, the authors find equivalent, easy-to-verify conditions for $\G$ to describe a maximal monotone graph according to relation~\eqref{impli} 
and we use the convergence result from~\cite{BuMaMa} later in this work when solving an approximative problem.

The most studied models in the theory of PDEs are of type \eqref{impli}. For the linear model $\S=2\nu_* \D$, $\nu_* \in (0, \infty)$, the existence theory for weak solutions was established in \cite{Leray} for the dimension $d=3$ and in the whole space, later extended to bounded domains and Dirichlet boundary condition in \cite{Hopf}. The non-linear explicit models of the type
\begin{equation*}
\S = 2 \nu(|\D|)\D, ~\text{ with }~ \nu:\R^+ \to \R^+,
\end{equation*}
where the mapping $\D\mapsto \S$ is monotone and continuous, were first studied in \cite{Lady1, Lady2, Lady3}. Especially, the case where
\begin{equation*}
\S = 2\nu_*(\alpha_* + |\D|^2)^{\frac{r-2}{2}} \D, ~\text{ with }~ r \ge 1, \nu_* >0, \alpha_* \in [0,\infty),
\end{equation*}
which for $\alpha_*=0$ is called the power-law model. Ladyzhenskaya established the existence of weak solution for $r\geq 11/5$ in 3-dimensional case (this corresponds to the possibility of testing by the weak solution and thus the use of the classical monotone operator theory). Despite their importance, these results were unsatisfactory since they even did not cover the case $r=2$. Nevertheless, it was the starting point, which finally gave birth to many new methods developed in the theory of non-linear PDEs, and which finally led to the complete theory for all $r> 6/5$ (the power which guarantees the compactness of the convective term in $3$-dimensional setting).

To mention the methods, we recall the \emph{higher differentiability method} from~\cite{MaNeRoRu} giving the existence for~$r>9/5$ for spatially periodic problem; the \emph{$L^\infty$~truncation method} from~\cite{FrMaSt} providing the existence for $r>8/5$ for perfect slip case or spatially periodic problem; up to the \emph{Lipschitz approximation method} in~\cite{DiRuWo} leading to the result for $r>6/5$ for Dirichlet boundary conditions. These results for explicit models were later systematically studied in the works~\cite{BGMS1, BGMS2}, that provide  results in the setting of \emph{maximal monotone graphs} for the same range of exponents (even more, the authors considered the setting of Orlicz spaces) and for the Navier slip boundary conditions. Last, we want to mention the recent result for~$r\le 6/5$ in~\cite{AbFe}, where the authors introduced a very generalized concept of solution suitable for parameters $r\le 6/5$ and proved the existence of such solution. Moreover, they showed that in case that smooth solution exists, their definition provides equivalent notion of solution. However, this concept of solution heavily relies on the fact that the graph~$\mathcal{A}$ comes as\footnote{It was already observed in~\cite{Ro}, that the subdifferential of a convex function generates maximal monotone graph defined by~\cite{Minty}. But in general, maximal monotone graph may not have a potential.}  \emph{a subdifferential of a convex potential}. Moreover, in case $r>6/5$, the concept of a solution introduced in~\cite{AbFe} is much weaker than the concept we deal with in the present paper. In addition, for graphs, which are maximally monotone but do not have a potential, such procedure cannot be used. Therefore, we do not consider the methods developed in~\cite{AbFe} here although they may be easily adapted also to the dynamic slip boundary conditions.

All of the models mentioned above relate the quantities inside the domain $\o$. Finally, the work~\cite{BM} also studies the (implicit) stick-slip condition acting on the boundary $\po$, i.e., for functions $\G$ and $\h$, such that
\begin{equation}\label{BM}
\G(\S, \D) = \0 \quad \text{in}~ \Omega, \quad \h(\s, \vv) = \0 \quad \text{on}~ \partial \Omega.
\end{equation}
This boundary condition can be (for suitable $\h$) viewed as an approximation of the Dirichlet boundary condition $\vv=\0$ on $\Gamma$. In~\cite{MaZa}, the authors studied all models from Figure~\ref{fig:responses} except the limiting ones, also in the case when they depend on the temperature. In the presented result, we use the setting similar to~\eqref{BM}, however, we importantly generalise the boundary condition~\eqref{NSrd} by the use of parameters~$\alpha$ and $\beta$, and by incorporating also the time derivative of the velocity on the boundary.

\subsection*{Dynamic slip model}

We study the phenomenon which has not attracted so much mathematical attention yet. It is called the dynamic slip and terms the response of certain fluids (typically polymers) to a sudden increase and consequent relaxation of the flow velocity, which results in the ``overshoot" of the slip on the boundary - the fluid first starts to slip very quickly, but after the sudden relaxation, it smoothly slows down and stabilizes its slip velocity.

We prove the existence for rather general classes of fluids, not only thanks to the range $r\in(6/5,\infty)$, we also do not prescribe any formulae for the graphs, nor assume the existence of a Borel measurable selection; we only require the maximality and monotonicity according to the definition of the maximal monotone graph (Definition \ref{maxmongrafA}). Moreover, we allow great generality thanks to the presence of the non-negative parameters $\alpha$ and $\beta$. In the following, we provide an explanation of their use in the typical combinations. In~\eqref{NSrd}, we obtain
\begin{subequations}\label{alfabeta}
\begin{align}
&\text{perfect slip} &&\text{if }\alpha=0 \text{ and } \beta = 0, \\
&\text{Navier's slip} &&\text{if }\alpha>0 \text{ and } \beta = 0, \\
&\text{no slip} &&\text{if }\alpha\to +\infty \text{ and } \beta = 0, \\
&\text{dynamic slip} &&\text{if }\alpha>0 \text{ and } \beta > 0.
\end{align}
\end{subequations}
Moreover, in the case when~$\alpha>0$, the structure of the graph~$\B$ plays its role and the model can describe many non-linear and implicit relations.

\subsection*{Motivation}

It is measured that under transient flow (by transient flow we mean the flow in the moment when the slip of the fluid starts), the slip velocity of the polymers exhibits relaxation behaviour in the sense that relaxation of polymer molecules next to solid walls is different compared to that in the bulk, and thus, delayed slip is observed. In such a case, the standard ``static" slip models (in our setting, corresponding to $\beta = 0$) do not follow the characteristics of the flow and therefore it is necessary to include the dynamical response of the fluid in the formulation of the model. In the dynamic slip models, the slip velocity might depend on the past deformation history undergone by the polymer, therefore the use of such dynamic models can explain basic slip rheological data, not otherwise explained by the use of static slip models.

We are not aware of any mathematical work which would analytically study such models. On the other hand, in polymer science, this effect is already well-known. First, referred to as ``retarded" slip boundary condition or ``memory" slip velocity model, it was studied in~\cite{PP} and since then, many other models were proposed, improving the original work and including some other effects. An illustrative overview on these models is presented in~\cite[Section 6.2]{H}. In fact, these works take into account reactions which occur on the boundary and in general relate the slip velocity of the polymer melt with the wall shear stress, the normal stress difference at the wall, the molecular weight, the molecular weight distribution, and the temperature, but also the reaction between bonded and free macromolecules at the interface.

Nonetheless, our model is macroscopic and these effects, as well as the molecular architecture of the polymer, can possibly be incorporated via coefficients~$\alpha$ and~$\beta$, and via appropriate definition of the graph $\B$. The dependence on these coefficients of models with simple geometries (the situation reduced to one dimensional flow) is demonstrated in the simulations in the next part, where the simple shear and the periodic flows are studied.

\subsection*{Result}

In Section~\ref{explicit}, we provide several explicit solutions in simplified geometry to illustrate the role of parameters in the dynamic slip boundary condition - these explicit solutions are computed and studied just for linear problems for simplicity. Next, in Section~\ref{spaces}, we fix the proper function space setting. The key difficulty is to incorporate the time derivative of the velocity on the boundary to a proper function space leading to a reasonable Gelfand triplet. When constructing the Gelfand triplets $V_r \hookrightarrow  H\equiv H^*  \hookrightarrow V_r^*$, we pay close attention to incorporating the boundary term, and the presence of its norm with the coefficient~$\beta$ also in the definition of the norm on the Hilbert space~$H$ (according to~\eqref{H}, the norm is $\|f\|^2_H:= \|f\|^2_{L^2(\o)}+ \beta \|\tr f\|^2_{L^2(\po)}$ for smooth $f$) is highly non-standard. Then, in Section~\ref{graphs}, we recall the basic concepts from the maximal monotone graph setting and finally in Section~\ref{NSr}, we precisely formulate the key result of the paper and provide its proof. 
Next,  for completeness, the Appendix~\ref{basis} is devoted to the study of the basis orthogonal in $V$ and orthonormal in $H$, which is used for defining the Galerkin approximations. Finally, to provide the complete information about the result also at the beginning of the manuscript, we formulate it here, but without any ambition to be rigorous - for precise formulation we refer to Section~\ref{NSr}.
\begin{theorem*}
For any sufficiently smooth data and maximal monotone $r$-graph with $r\in (6/5, \infty)$ there exists a global-in-time weak solution to the system \eqref{problem}. Moreover, the solution satisfies the energy inequality and for $r\in[11/5, \infty)$ the energy equality.
\end{theorem*}

\subsection*{Notation}\label{notation}
\noindent\textit{Domains.}
For $d \geq 1$, we consider an open Lipschitz set $\o \subset \R^d$, and for $t \in (0,T]$, we denote $Q_t:=[0,t) \times \o$ and $\Gamma_t := [0,t) \times \po$. Also, we use simply~$Q$ and~$\Gamma$ for~$Q_T$ and~$\Gamma_T$, respectively (this does not concern the part with explicit examples).

\noindent\textit{Functions.}
No explicit distinction between spaces of scalar- and vector-valued functions will be made, but we employ small boldfaced letters to denote vectors and bold capitals for tensors. Outward normal vector is denoted by $\n$, and for any vector-valued function $\z:\po\to \R^d$, the symbol $\z_\tau$ stands for the projection to the tangent plane, i.e., $\z_\tau:= \z- (\z\cdot\n) \n$. If it is clear from the context, we denote the traces of Sobolev functions like the original functions, and if we want to emphasize it, we use the symbol $``\tr\!\!"$. Also, we do not relabel the original sequence when selecting a subsequence. The symbols $``\cdot"$ and $``:"$ stand for the scalar product of vectors or tensors, respectively, and $``\otimes"$ signifies the tensor product. In a time-space domain, the standard differential operators, like gradient ($\nabla$) and divergence ($\diver$), are always related to the spatial variables only. Also, we use standard notation for partial ($\partial_{\cdot}$ or $\partial_{\cdot \cdot}$) and total ($\frac{d}{d \cdot}$) derivatives or just the symbol `$'$' for the derivative of function of one variable. The Kronecker delta is denoted by $\delta_{i,j}$. Generic constants, that depend just on data, are denoted by~$C$ and may vary line to line.

\noindent\textit{Spaces.}
For a Banach space~$X$, its dual is denoted by $X^*$. For $x \in X$ and $x^*\in X^*$, the duality is denoted by $\langle x^*, x \rangle_X$. For $r \in [1,\infty]$, we denote $(L^r(\o), \|\!\cdot\!\|_{L^r(\o)})$ and $(W^{1,r}(\o),\|\!\cdot\!\|_{W^{1,r}(\o)})$ the corresponding Lebesgue and Sobolev spaces with norms. Bochner space is designated by $L^r(0,T;X)$.
For $r \in [1, \infty]$, we set
\begin{align*}
W^{1,r}_{\diver}(\o) &:= \overline{\bigl\{\f \in \C^{0,1}(\overline{\o});\, \diver \f =0 \text{ in }\o \bigr\}}^{\|\cdot\|_{W^{1,r}(\o)}},\\
W^{1,r}_{\n}(\o) &:= \overline{\bigl\{\f \in \C^{0,1}(\overline{\o});\, \f \cdot \n = 0 \text{ on } \po \bigr\}}^{\|\cdot\|_{W^{1,r}(\o)}},\\
W^{1,r}_{\n,\diver}(\o) &:= \overline{\bigl\{\f \in \C^{0,1}(\overline{\o}); \,\f \cdot \n = 0 \text{ on } \po, \diver \f =0 \text{ in }\o \bigr\}}^{\|\cdot\|_{W^{1,r}(\o)}},\\
\C([0,T];X) &:= \{f \in L^{\infty}(0,T;X);\, [0,T] \ni t^n \! \to \! t \Rightarrow f(t^n) \! \to \! f(t) \text{ strongly in $X$}\}, \\
\C_{w}([0,T];X) &:= \{f \in L^{\infty}(0,T;X);\, [0,T] \ni t^n \! \to \! t \Rightarrow f(t^n) \! \tow \! f(t) \text{ weakly in $X$}\}.
\end{align*}
%

\section{Explicit examples}\label{explicit}

We list several prototypes of the problem we want to solve. We provide two explicit examples (without the use of the maximal monotone graphs), where in simple situations, we clearly demonstrate the use of the dynamic slip boundary condition. Analytical computations are sketched and supported by numerical simulations.

The solutions are found more or less in the same way as for the classical slip boundary condition with one proviso - the basis in which we construct the solution corresponds to a different boundary condition. This however changes the properties of the solution drastically, in particular (and it will be also evident from computation), the first few eigenvalues and eigenfunctions are of most importance to give the character of the flow.

The general setting is the same for both examples. For simplicity, both flows act in one direction only, and they differ by the use of the boundary conditions and assumption on pressure, which determines the regime of the flow. In the first case, we talk about the flow induced by moving boundary, whereas in the second case, the pressure initiates a time-periodic flow.

For $h, T >0$, define $Q:=(0, T)\times \R^2\times (0,h)$ and consider the Navier--Stokes problem for an incompressible fluid in a three-dimensional domain, given by the system
\begin{subequations}\label{geneprob}
\begin{align}
\diver \vv &=0 &&\text{in } Q,\\
\partial_t \vv+\diver(\vv \otimes \vv)- \diver \S &=-\nabla p  &&\text{in } Q, \\
\S &= 2 \D = \nabla \vv + \nabla \vv^T &&\text{in } Q,\\
\sg &= \vv &&\text{in } (0,T)\times\mathbb{R}^2\times\{0,h\}.
\end{align}
We look for a solution to the simple shear which is represented by a scalar function $u:(0,T)\times (0,h) \to \R$,
\begin{equation}\label{defu}
\vv(t,\x) := (u(t,x), 0,0),
\end{equation}
\end{subequations}
where variable $x$ of $u$ corresponds to $x_3$ ($\x=(x_1,x_2,x_3)$) of $\vv$. Due to the definition~\eqref{defu}, the condition $\diver \vv = 0$ is automatically satisfied.

\subsection{Flow induced by moving boundary}

For given $\delta$, $0<\delta \ll 1$, the flow between two infinite planes is induced by moving one of them, $\R^2\times\{h\}$, with the velocity $V_\delta(t):=\min\{t/ \delta, 1\}$. It means that for small times, the upper plane accelerates really quickly, and after reaching velocity equal to $1$, it suddenly relaxes and continues to move with this constant velocity. The lower plane, $\R^2\times\{0\}$, does not move. Also, the pressure is only a function of time,
\begin{subequations}\label{movbound}
\begin{equation}
\nabla p= \0.
\end{equation}
We consider the following initial and boundary conditions, representing zero velocity of the fluid everywhere at the beginning as well as on the lower boundary for all times, whereas the velocity on the upper part of the boundary is expressed as a difference between the actual velocity of the fluid and the velocity of the moving plane,
\begin{align}
\vv &= \0 &&\text{in } \{0\}\times\mathbb{R}^2\times (0,h), \\
\vv &= \0 &&\text{in } (0,T)\times\mathbb{R}^2\times\{0\}, \\
\alpha [\sg - (V_\delta,0,0)] + \beta \partial_t [\vv - (V_\delta,0,0)] - \s &= \0 &&\text{in } (0,T)\times\mathbb{R}^2\times\{h\}.
\end{align}
\end{subequations}
Especially, we aim to study the dependence of solution on $\alpha$ and $\beta$ if $\delta \ll 1$ (this condition enhances the sudden acceleration of the boundary at the initial moment). We can reformulate the system \eqref{geneprob}--\eqref{movbound} in terms of function $u$,
\begin{subequations}\label{movboundu}
\begin{align}
\partial_t u(t,x) - \partial_{xx} u(t,x) &=0 &&\text{in } (0,T)\!\times\!(0,h), \label{eqmovbound}\\
u(0,x) &= 0 &&\text{in } (0,h), \label{ini0}\\
u(t,0) &= 0 &&\text{in } (0,T),\label{00} \\
\alpha [u(t,h) - V_\delta(t)] + \beta \partial_t [u(t,h) - V_\delta(t)] +  \partial_x u(t,h) &= 0 &&\text{in } (0,T).\label{bound}
\end{align}
\end{subequations}

We wish to construct a weak solution to \eqref{movboundu} in terms of Fourier series. The crucial step to do so is to properly define the function space for $u$ and properties of its basis. To insure \eqref{00}, let
\begin{equation}\label{simpleV}
V:=\{v \in W^{1,2}(0,h); v(0)=0\}, ~\langle v_1, v_2\rangle_V := \int_0^h \! v_1 v_2 \d x + \beta (v_1 v_2)(h),
\end{equation}
be the function space with duality and take a basis $\{u_i\}_{i\in \N}$ of $V$ which fulfills
\begin{subequations}\label{uprob}
\begin{align}
-u_i''(x)&=\lambda_i^2 u_i(x) &&\text{for } x \in(0,h), \label{lambda} \\
\alpha u_i(h) + u_i'(h) &= \lambda_i^2 \beta u_i(h), &&\text{and} \label{lambdabeta}\\
\left( u_i, u_j\right)_V:=\int_0^h u_i u_j \d x + \beta (u_i u_j)(h)&= \delta_{i,j} &&\text{for all } i,j \in \N, \label{dirac}
\end{align}
\end{subequations}
where we let~\eqref{dirac} define the scalar product in~$V$, and then the basis is orthonormal in~$V$.

We first prove existence of such basis and study the properties of the sequence $\{\lambda_i\}_{i\in\N}$. After that, we use this information to demonstrate the existence of the dynamic slip phenomenon as well as the fact that this effect vanishes as $\beta$ tends to~$0$.

From \eqref{lambda} we know that $u_i$ is of the form $u_i(x) = A_i \sin(\lambda_i x) + B_i \cos(\lambda_i x)$ (for $A_i$, $B_i$ constants), however, due to the condition $u_i(0) = 0$ (according to the definition of $V$ \eqref{simpleV}), this reduces to
\begin{equation} \label{ui}
u_i(x) = A_i \sin(\lambda_i x).
\end{equation}
As we generate a basis, without loss of generality we can assume that $A_i, \lambda_i >0$. Also, using \eqref{lambdabeta} for this $u_i$ we get the condition on $\lambda_i$,
\begin{equation}\label{condli}
(\alpha - \beta \lambda_i^2)\sin(\lambda_i h) = -\lambda_i \cos(\lambda_i h).
\end{equation}
To have an idea about the arrangement of the eigenvalues $\{\lambda_i\}_{i\in\N}$ within $\mathbb{R}^+$, for every $j \in \N_0$, we define an auxiliary function $f_j: [0,2\pi/h] \to \mathbb{R}$ as
\begin{equation*}
f_j(y) := \left(\alpha - \beta \left(y+j\frac{2\pi}{h}\right)^2\right)\sin(y h)  + \left(y+j\frac{2\pi}{h}\right) \cos(y h).
\end{equation*}
For every $j \in \N_0$, there exist at least two solutions to $f_j(y)=0$. In fact, there are at most two, as the following explains,
\begin{equation}\label{solli}
f_j(y) = 0 ~~\Longleftrightarrow~~ \cot(y h) = \beta \left(y+j\frac{2\pi}{h}\right) - \frac{\alpha}{y+j\frac{2\pi}{h}} .
\end{equation}
Here, the function on the right hand side is increasing for every $j$ and cotangent is decreasing on $(0,\frac{\pi}{h})$ and $(\frac{\pi}{h}, \frac{2\pi}{h})$. Therefore,
\begin{equation}\label{locli}
\text{for every } i \in \mathbb{N} \text{ there exists a unique } \lambda_i \in \left((i-1 )\frac{\pi}{h}, i\frac{\pi}{h}\right) \text{ solving \eqref{condli}}.
\end{equation}
Thanks to \eqref{condli} and \eqref{solli},
\begin{equation}\label{negcot}
\beta \lambda_i^2 - \alpha< 0 ~\Leftrightarrow~ \beta \left( i-\frac12 \right)\frac{\pi}{h} - \frac{\alpha}{  \left( i-\frac12 \right)\frac{\pi}{h}}<0 ~\Leftrightarrow~ (2i-1)^2 < \frac{4h^2 \alpha}{\pi^2 \beta}.
\end{equation}
Finally, using  \eqref{condli} in \eqref{dirac}, we obtain the formula for $A_i$,
\begin{equation}\label{condAi}
A_i = \left( \frac{h}{2} + \frac{\alpha + \beta \lambda_i^2}{2\lambda_i^2} \sin^2(\lambda_i h) \right)^{-\frac12}.
\end{equation}
Therefore, the basis $\{u_i\}_{i\in\N}$ exists and is prescribed by \eqref{ui} such that \eqref{condli}, \eqref{locli} and \eqref{condAi} hold.

Next, we look for a weak solution to our system $w(t,x)$ defined by means of this basis, i.e., $w(t,x) = \sum_{i=1}^\infty c_i(t) u_i(x)$ so that for all $\varphi \in W^{1,2}(0,h)$ fulfilling $\varphi(0)=0$ and almost all $t\in (0,T)$, the integral formulation of \eqref{eqmovbound} holds,
\begin{equation*}
\int_0^h \partial_t w \varphi + \partial_x w \varphi' \d x + \left[\alpha (w-V_\delta)\varphi + \beta \partial_t (w-V_\delta)\varphi\right]_{x=h} = 0,
\end{equation*}
where we first used integration by parts and consequently the boundary conditions \eqref{00} and \eqref{bound} (for simplicity, we a~priori assume that our solution is smooth enough so that the integral exists and we can substitute $\langle \pt w, \varphi \rangle_V = \int_0^h \pt w \varphi \d x + \beta (\pt w \varphi)(h)$ according to \eqref{simpleV}).

Now, we set $\varphi := u_j$ and use~\eqref{uprob} to obtain
\begin{equation*}
c_j'(t) + \lambda_j^2 c_j(t) =  (\alpha V_\delta(t) + \beta V_\delta'(t) )u_j(h) ,
\end{equation*}
which, completed with the initial condition $c_j(0) = 0$ (due to \eqref{ini0}) and using the definition of $V_\delta$, turns into
\begin{align}
c_j^{\delta_-}(t) &= \frac{u_j(h)}{\delta \lambda_j^2} \left( \alpha t+ (\beta\lambda_j^2 - \alpha)\frac{e^{-\lambda_j^2 t} -1}{-\lambda_j^2} \right) \simeq \beta u_j(h) \frac{t}{\delta} &&t < \delta \ll 1, \label{td} \\
c_j^{\delta_+}(t) &= \frac{u_j(h)}{\lambda_j^2} \left( \alpha + (\beta\lambda_j^2 - \alpha)e^{-\lambda_j^2 t}\frac{e^{\lambda_j^2 \delta} -1}{\lambda_j^2 \delta} \right) \simeq \frac{u_j(h)}{\lambda_j^2} \left( \alpha + \frac{\beta\lambda_j^2 - \alpha}{e^{-\lambda_j^2 t}} \right) &&\hspace{.6cm} \delta \ll 1 \label{dt},
\end{align}
having defined $c_j^{\delta_-}(t) := c_j(t)\chi_{t < \delta}$ and $c_j^{\delta_+}(t) := c_j(t)\chi_{t\ge \delta}$. It is not difficult to check that $\lim_{t \to \delta_-} c_j^{\delta_-}(t) = c_j^{\delta_+}(\delta)$ and that $c_j(t) = c_j^{\delta_-}(t)+c_j^{\delta_+}(t)$ is continuous. Finally,
\begin{align}\label{wtx}
\!\!w(t,x) = \left\{
\begin{aligned}
&\sum_{i=1}^\infty \frac{1}{\delta \lambda_i^2}\left( \alpha t+ (\beta\lambda_i^2 - \alpha)\frac{e^{-\lambda_i^2 t} -1}{-\lambda_i^2} \right) u_i(h) u_i(x), &&t \in (0, \delta), \\
&\sum_{i=1}^\infty  \frac{1}{\lambda_i^2} \left( \alpha + (\beta\lambda_i^2 - \alpha)e^{-\lambda_i^2 t}\frac{e^{\lambda_i^2 \delta} -1}{\lambda_i^2 \delta} \right) u_i(h) u_i(x), &&t \in (\delta, T).
\end{aligned}\right.
\end{align}
For us, it is important to study the behaviour of this solution on the boundary where $x=h$ and for $\delta \to 0_+$. Applying \eqref{condli}, \eqref{ui}, and \eqref{condAi} in \eqref{wtx} for $x=h$, and proceeding with $\delta \to 0_+$, we get that
\begin{equation}\label{wth}
w(t,h) \to \sum_{i=1}^\infty \frac{2}{h(\lambda_i^2 + (\beta \lambda_i^2 - \alpha)^2) + \alpha + \beta \lambda_i^2} \left(\alpha + (\beta\lambda_i^2 - \alpha)e^{-\lambda_i^2 t} \right).
\end{equation}
In \eqref{wth}, we already neglected the first part of \eqref{wtx}, where $t\in(0,\delta)$, and also used that $\lim_{\delta \to 0_+}(e^{\lambda_i^2 \delta} -1)/(\lambda_i^2 \delta) =1$.

Once we have a weak solution, we want to check whether it converges to the stationary one, $\bar{w}(x) = Ax=\sum_{i=1}^\infty \bar{c}_i u_i(x)$, as $t =T \to \infty$. Here, $A$ is a constant and $\bar{c}_i$ are coefficients that satisfy
\begin{equation*}
\bar{c}_i = \left(\sum_{j=1}^\infty \bar{c}_j u_j, u_i \right)_V = \left(\bar{w} , u_i \right)_V = \frac{u_i(h)}{\lambda_i^2}A (\alpha h + 1).
\end{equation*}
The difference $w(t,x) - \bar{w}(x)$ vanishes for $t \to \infty$ if $A := \frac{\alpha}{\alpha h + 1}$,
\begin{equation*}
w(t,x) \to \bar{w}(x)= \frac{\alpha}{\alpha h + 1}x.
\end{equation*}
This limiting solution corresponds to the stationary solution with the standard Navier slip response. Therefore, to study the dynamic slip phenomenon, one pays attention to the difference of these solutions on the boundary for small (although relevant) times,
\begin{equation}\label{diff}
w(t,h) - \bar{w}(h) = \sum_{i=1}^\infty \frac{u_i^2(h)}{\lambda_i^2} \left( (\beta\lambda_i^2 - \alpha )e^{-\lambda_i^2 t} \right) \text{ for } t\geq \delta.
\end{equation}
From the relation \eqref{diff} one can see that for small times, the impact of the several first terms is much more important than that of the terms for larger values of $\lambda_i$ (note that the sequence $\{\lambda_i\}_{i\in\N}$ is increasing as it is arranged according to \eqref{locli}). Also, as discussed in \eqref{negcot}, these first terms (their number depends on $\alpha, \beta$ and $h$) in the sum \eqref{diff} can be negative, whereas for larger values of $i$ they become positive.

\subsection*{Simulations}

The importance of the number of negative terms in \eqref{diff} can be demonstrated using the computational software. In what follows, we numerically computed the sequence $\{\lambda_i\}_{i=1}^{10}$ and present the corresponding graphs of the solution $1 - w(t,h)$ (which is $V_\delta(t) - w(t,h)$ for $\delta \to 0_+$) in several situations; namely for $\alpha/ \beta \in \{1/4, 20, +\infty\}$. Then, we compare the graphs for fixed $\alpha$ and three different values of $\beta$ with the graph of the stationary solution $1-\bar{w}(h)$ (which only depends on $\alpha$ and $h$ and therefore is the same for the three solutions). In all simulations, we fixed the constant $h=\pi$.

In Figure \ref{fig:compareslip}, we are interested in the response for short times, therefore, we set $T=1$. From \eqref{negcot} we can read that the number of terms for which $\beta\lambda_i^2 - \alpha <0$ is
\begin{equation}\label{numneg}
\NN:=\left| \left\{i; \beta\lambda_i^2 - \alpha <0 \right\}\right| = \max\left\{i; i< \sqrt{\frac{\alpha}{\beta}}+\frac{1}{2}\right\}.
\end{equation}

First simulation corresponds to $\alpha = 1$ and $\beta = 4$. Due to \eqref{numneg}, $\NN=0$ and we can see that every term in \eqref{diff} (which correspond to the difference $1-\bar{w}(h) - (1 - w(t,h))$) is positive and therefore the graph of $1 - w(t,h)$ monotonically increases while approaching the stationary solution $1-\bar{w}(h)$.

In the second situation, we used $\alpha = 10$ and $\beta = 0.5$. According to \eqref{numneg}, $\NN=4$. This combination allows to model the dynamic slip phenomenon, as it clearly demonstrates that the behaviour of the fluid on the boundary is not monotone. Indeed, the relative velocity  $1 - w(t,h)$ first continues to increase and subsequently slows down and starts decreasing to approach the stationary solution $1-\bar{w}(h)$.

Finally, we used $\alpha = 10$ and $\beta = 0$. Such a choice corresponds to the Navier slip situation, as the effect of the time derivative in \eqref{bound} is cancelled. Also, consistently with \eqref{numneg}, $\NN \to +\infty$ and all terms in \eqref{diff} are negative. This results in significant jump at origin and subsequently, the graph of the solution $1 - w(t,h)$ immediately monotonically decreases as it approaches its stationary solution.

\begin{figure}[h]
\begin{minipage}{0.65\textwidth}
\centering
\includegraphics[width=\textwidth]{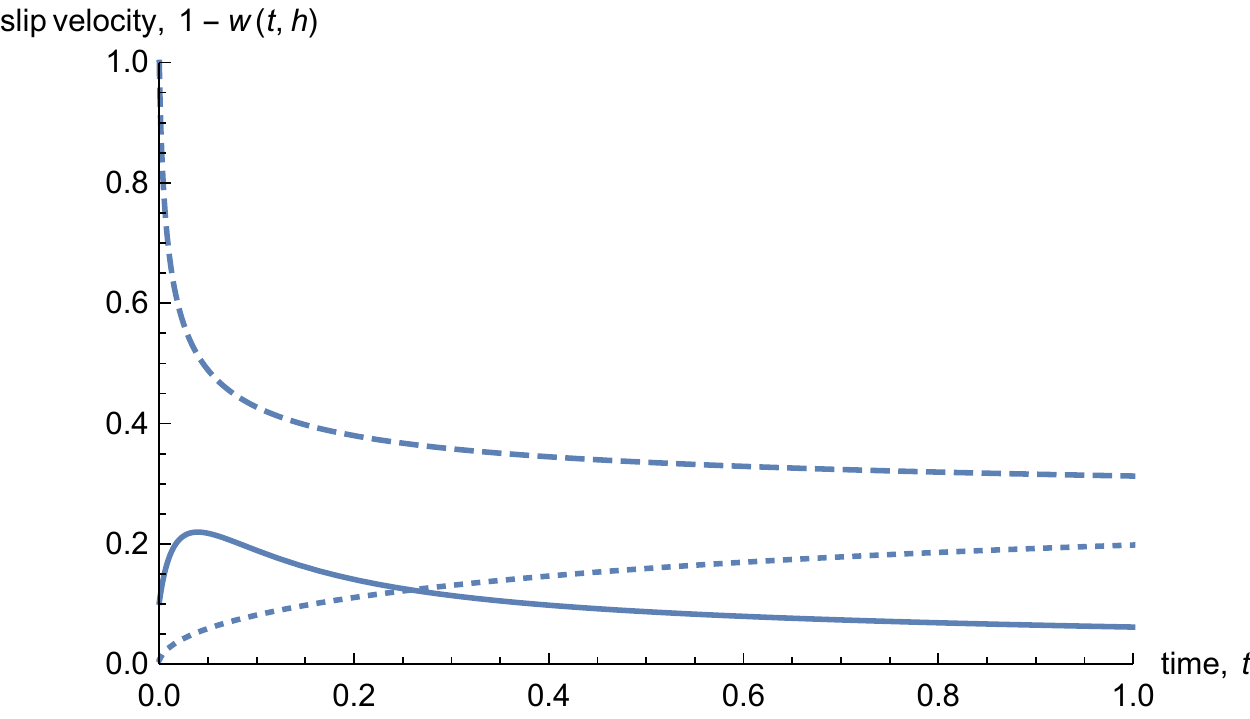}
\end{minipage}%
\begin{minipage}{0.35\textwidth}
\captionsetup{width=\linewidth}
\caption{Comparison of the slip velocities $1 - w(t,h)$ for different values of $\alpha$ and $\beta$. For the dotted style $(\alpha, \beta, \NN) = (1,4,0)$; for the full line $(\alpha, \beta, \NN) = (10,0.5, 4)$ (two dynamic slips) and for the dashed style $(\alpha, \beta, \NN) = (10,0, + \infty)$ (Navier's slip).}
\label{fig:compareslip}
\end{minipage}
\end{figure}

\begin{figure}[h]
\begin{minipage}{0.65\textwidth}
\centering
\includegraphics[width=\textwidth]{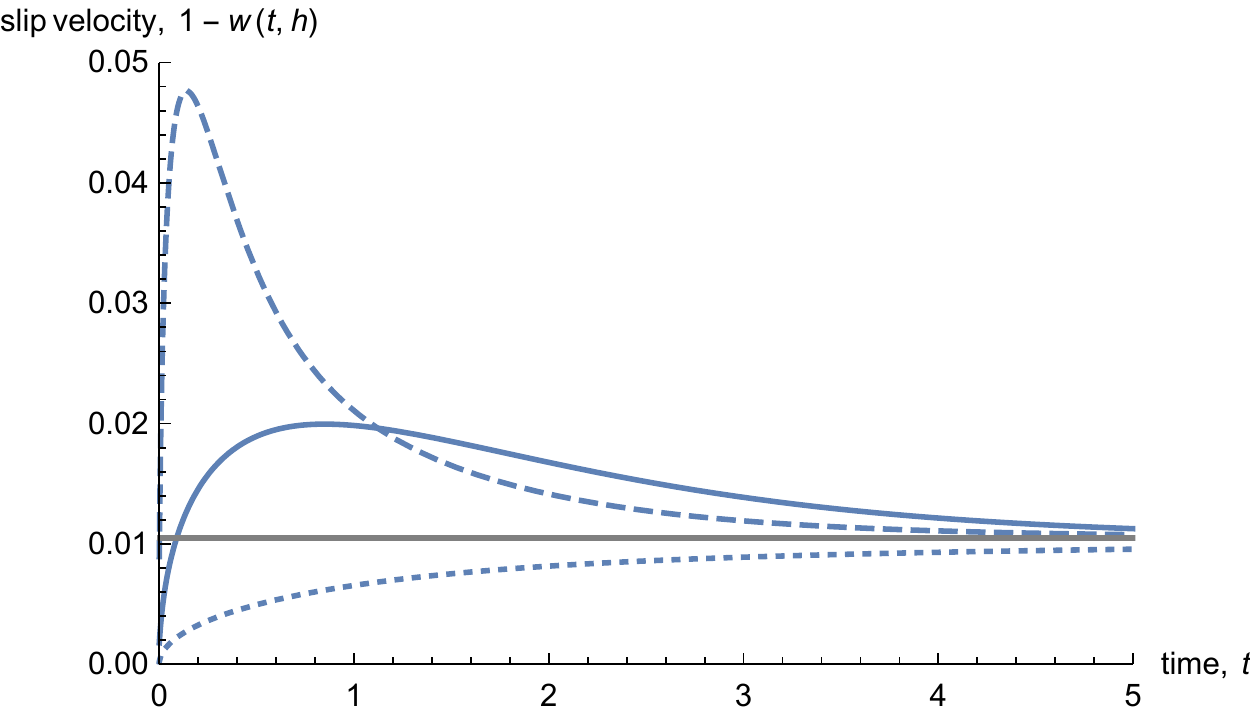}
\end{minipage}%
\begin{minipage}{0.35\textwidth}
\captionsetup{width=\linewidth}
\centering
\caption{Comparison of the slip velocities $1 - w(t,h)$ for fixed value of $\alpha=30$ and different values of $\beta$. For the dashed style $(\beta, \NN) = (5,2)$; for the full line $(\beta, \NN) = (30,1)$ and for the dotted line $(\beta, \NN) = (150,0)$.}
\label{fig:compareconv}
\end{minipage}
\end{figure}

These simulations, especially comparison of the second and the third one, clearly explain why in our modelled situation with significant impulse in the beginning (which in the picture corresponds to the jump at $t=0$) and sudden relaxation thereafter, it is much  more natural to expect the smooth dynamic slip response than the sharp Navier slip.

The second simulation indicates the importance of the value of $\beta$ for the dynamic slip, as well as the convergence property of the solutions. For this reason, we set $T=5$, which is large enough to see the converging tendency. We fix $\alpha=30$ and compare the graphs of $1 - w(t,h)$ for $\beta \in \{5,30,150\}$ against the stationary solution $\bar{w}(h) = \alpha h/(\alpha h +1)$. Since the stationary solution is independent of $\beta$, we can also see that all three graphs converge to this stationary solution. The results are presented in the Figure \ref{fig:compareconv}.

\subsection{Periodic flow induced by pressure}

In the second example, we consider the pressure of the form
\begin{subequations}\label{perflow}
\begin{equation}
p= x_1 \cos \left( \frac{2 \pi t}{T} \right),
\end{equation}
which induces a time-wise periodic flow, so that the initial and boundary conditions are
\begin{align}
\vv(0, \x) &= \vv(T, \x) &&\text{in } \mathbb{R}^2\times (0,h), \\
\vv &= \0 &&\text{in } (0,T)\times\mathbb{R}^2\times\{0\}, \\
\alpha \sg + \beta \partial_t \vv - \s &= \0 &&\text{in } (0,T)\times\mathbb{R}^2\times\{h\}.
\end{align}
\end{subequations}
By means of $u$, we can reformulate the system \eqref{geneprob}, \eqref{perflow} to
\begin{subequations}\label{perflowu}
\begin{align}
\partial_t u(t,x) - \partial_{xx} u(t,x) &= - \cos \left( \frac{2 \pi t}{T} \right)  &&\text{in } (0,T)\times(0,h), \label{1scos}\\
u(0,x) &= u(T,x) &&\text{in } (0,h), \label{peri}\\
u(t,0) &= 0 &&\text{in } (0,T),\label{u00} \\
\alpha u(t,h) + \beta \partial_t u(t,h) + \partial_x u(t,h) &= 0 &&\text{in } (0,T).\label{uth}
\end{align}
\end{subequations}

Similarly as in the previous part, we wish to construct a weak solution to \eqref{perflowu} and thanks to \eqref{u00} (which is identical to \eqref{00}), we can work with exactly the same space $V$ as before, defined in \eqref{simpleV}, and its base $\{u_i \}_{i \in \N}$ satisfying \eqref{uprob} and prescribed by \eqref{ui} ($u_i(x) = A_i \sin(\lambda_i x)$) such that \eqref{condli}, \eqref{locli} and \eqref{condAi} hold.

The essence of the problem now lies in finding the coefficients $c_i$, where our weak solution to \eqref{perflowu} is again of the form $w(t,x)= \sum_{i=1}^\infty c_i(t)u_i(x)$ and for all $\varphi \in W^{1,2}(0,h)$ and almost all $t\in(0,T)$ satisfies
\begin{equation}
\int_0^h \pt w \varphi + \partial_x w \varphi' \d x + \left[(\alpha w+ \beta \pt w)\varphi\right]_{x=h} = -\cos \left( \frac{2 \pi t}{T} \right) \int_0^h \varphi \d x,
\end{equation}
which we obtained by multiplying \eqref{1scos} by $\varphi$, using integration by parts, \eqref{u00}, and \eqref{uth}. Now, we set $\varphi := u_j$ for $j \in \N$ and use the definition of $w$, the orthonormality of the basis \eqref{dirac}, \eqref{lambdabeta} and \eqref{ui} to get
\begin{equation}
c_j'(t) + \lambda_j^2 c_j(t) = - \cos \left( \frac{2 \pi t}{T} \right) \int_0^h u_j \d x = \frac{A_j}{\lambda_j} (\cos (\lambda_j h)-1) \cos \left( \frac{2 \pi t}{T} \right).
\end{equation}
Solving this equation, using that $c_j(0) = c_j(T)$ and that
\begin{align*}
\int_0^t \! e^{\lambda_j^2 \tau} \cos \left( \frac{2 \pi \tau}{T} \right) \!\d \tau\!= \frac{2 \pi T}{4 \pi^2 + \lambda_j^4 T^2} \left[e^{\lambda_j^2 t} \sin \left( \frac{2 \pi t}{T} \right) + \frac{\lambda_j^2 T}{2 \pi} \left(e^{\lambda_j^2 t} \cos \left( \frac{2 \pi t}{T} \right)-\!1 \!\right) \right],
\end{align*}
we obtain the formula for the initial condition $c_j(0)$ and for $c_j(t)$,
\begin{subequations}\label{cj}
\begin{align}
c_j(0) &= c_j(T)= \frac{\lambda_j T^2}{4\pi^2 + \lambda_j^4 T^2} A_j (\cos (\lambda_j h)-1), \label{cj0}\\
c_j(t) &= \frac{2 \pi T}{4 \pi^2 + \lambda_j^4 T^2} \frac{A_j}{\lambda_j}  (\cos(\lambda_j h)-1) \left[\sin \left( \frac{2 \pi t}{T} \right) + \frac{\lambda_j^2 T}{2 \pi} \cos \left( \frac{2 \pi t}{T} \right) \right]. \label{cjt}
\end{align}
\end{subequations}
To sum up, using \eqref{cjt}, \eqref{ui} and \eqref{condAi}, the solution satisfies
\begin{align}
w(t,x) &=  \sum_{i=1}^\infty c_i(t) A_i \sin(\lambda_i x),\nonumber \\
\partial_x w(t,h) &=  \sum_{i=1}^\infty c_i(t) \lambda_i A_i   \cos(\lambda_i h)\label{xwth},
\end{align}
where the spatial derivative represents the wall shear stress, which is the quantity that we finally compare.

Once we have the weak solution $w(t,x)$, similarly as before, we want to compare it with some reference solution - in this case, the periodic solution with the Dirichlet boundary condition, i.e., $\bar{w}(t,x) = \sum_{i=1}^\infty \bar{c}_i(t) \bar{u}_i(x)$, such that $\bar{w}(t,x)$ is a weak solution to
\begin{subequations}\label{perflowbar}
\begin{align}
\partial_t u(t,x) - \partial_{xx} u(t,x) &= - \cos \left( \frac{2 \pi t}{T} \right)  &&\text{in } (0,T)\times(0,h), \label{u1scos}\\
u(0,x) &= u(T,x) &&\text{in } (0,h), \label{uperi}\\
u(t,0) &= u(t,h) = 0 &&\text{in } (0,T).\label{u0h0}
\end{align}
\end{subequations}
Then, for every $i\in\N$, the eigenfunctions $\bar{u}_i(x)$ of $W^{1,2}(0,h)$ with their eigenvalues $\bar{\lambda}_i$ solve
\begin{subequations}
\begin{align}
-\bar{u}_i''(x)&=\bar{\lambda}_i^2 \bar{u}_i(x) &&\text{for } x \in(0,h), \label{uurovn} \\
\bar{u}_i(0) &= \bar{u}_i(h) =0 &&\text{and} \label{uuokr}\\
\int_0^h \bar{u}_i \bar{u}_j \d x &= \delta_{i,j} &&\text{for all } i,j \in \N, \label{uudirac}
\end{align}
\end{subequations}
to form a basis in $W^{1,2}_0(0,h)$, and $\bar{c}_i(t)$ is computed using $\bar{u}_i(x)$. From \eqref{uurovn} we know that the eigenfunctions are of the form
\begin{equation}\label{barui}
\begin{aligned}
\bar{u}_i(x) = \bar{A}_i \sin(\bar{\lambda}_i x) + \bar{B}_i \cos (\bar{\lambda}_i x), \text{ where } \\
\bar{B}_i = 0,~\bar{\lambda}_i = \frac{i \pi}{h} ~\text{ and }~\bar{A}_i = \sqrt{\frac{2}{h}},
\end{aligned}
\end{equation}
using \eqref{uuokr} and \eqref{uudirac}. Altogether,
\begin{equation}\label{uui}
\bar{u}_i(x) = \bar{A}_i \sin (\bar{\lambda}_i x)= \sqrt{\frac{2}{h}} \sin\left(\frac{i \pi}{h} x\right).
\end{equation}
Then, for all $\varphi \in W^{1,2}_0(0,h)$ and almost all $t\in(0,T)$, $\bar{w}$ satisfies
\begin{equation}
\int_0^h \pt \bar{w} \varphi + \partial_x \bar{w} \varphi' \d x = -\cos \left( \frac{2 \pi t}{T} \right) \int_0^h \varphi \d x,
\end{equation}
and for $\varphi:=\bar{u}_j$, $j \in \N$, repeating a very similar procedure like before, we obtain that $\bar{c}_i(t)$ solves
\begin{equation}
\bar{c}_i'(t) + \bar{\lambda}^2_i \bar{c}_i(t) = - \cos \left( \frac{2 \pi t}{T} \right) \int_0^h \bar{u}_i \d x =
\begin{cases}
2 \frac{\bar{A}_i}{\bar{\lambda}_i} \cos \left( \frac{2 \pi t}{T} \right), &i=2k+1, \\
0, &i=2k,
\end{cases}
\end{equation}
for $k \in \N$. Then,
\begin{equation}\label{barci}
\bar{c}_i(t) =
\begin{cases}
\frac{2 \pi T}{4 \pi^2 + \bar{\lambda}_i^4 T^2} \frac{2 \bar{A}_i}{\bar{\lambda}_i} \left[ \sin\left( \frac{2 \pi t}{T} \right) + \frac{\bar{\lambda}_i^2 T}{2 \pi} \cos \left( \frac{2 \pi t}{T} \right) \right], &i=2k+1, \\
0, &i=2k,
\end{cases}
\end{equation}
and we can finally write the formulae similar to those for $w(t,x)$, but incorporating \eqref{barci} and \eqref{barui},
\begin{align}
\bar{w}(t,x) &=  \sum_{i=1}^\infty \bar{c}_i(t) \bar{A}_i \sin(\bar{\lambda}_i x),\nonumber\\
\partial_x \bar{w}(t,h) &=  \sum_{i=1}^\infty \bar{c}_i(t) \bar{\lambda}_i\bar{A}_i \cos(\bar{\lambda}_i h). \label{xbwth}
\end{align}

\begin{figure}[h]
\begin{minipage}{0.6\textwidth}
\centering
\includegraphics[width=\textwidth]{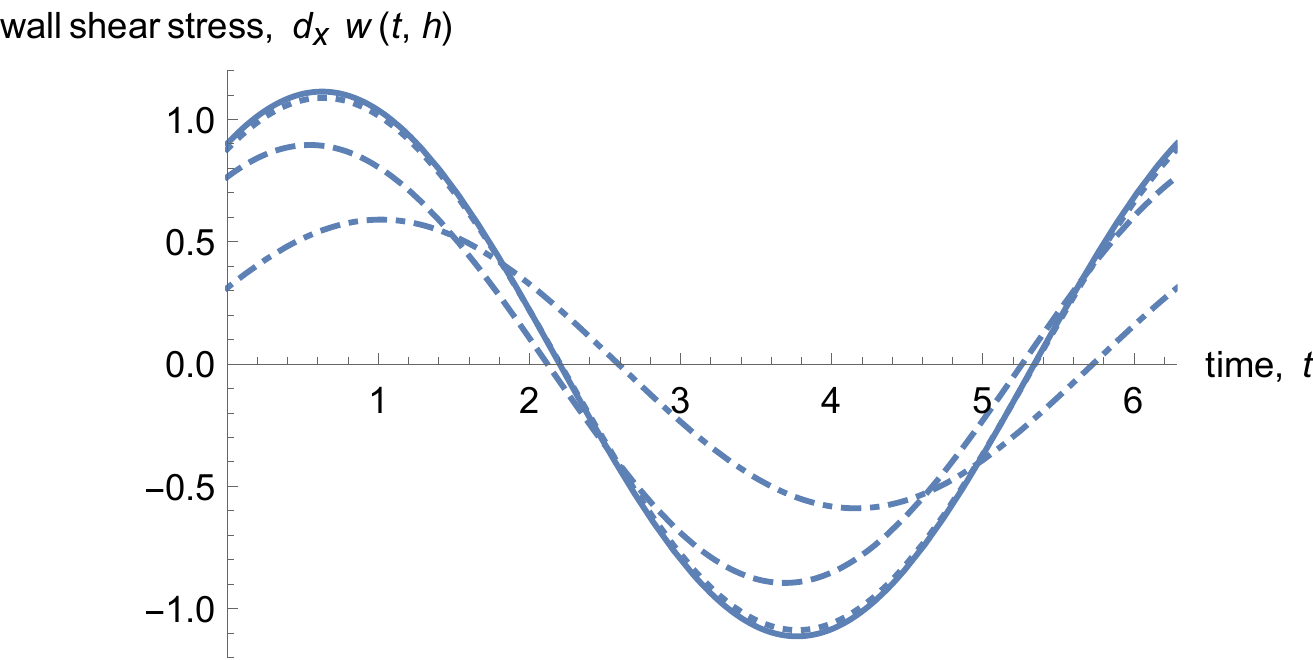}
\end{minipage}\quad
\begin{minipage}{0.35\textwidth}
\captionsetup{width=\linewidth}
\centering
\caption{Comparison of the wall shear stresses $\partial_x w(t,h)$ for fixed value of $\alpha=1$ and three different values of $\beta \in \{0.1, 4.2, 100\}$. For the dot-dashed style $\beta = 0.1$; for the dashed style $\beta = 4.2$ and for the dotted style $\beta = 100$. The full line corresponds to the Dirichlet solution, which is independent of $\alpha$ and $\beta$.}
\label{fig:compareperi}
\end{minipage}
\end{figure}

\begin{figure}[h]
\begin{minipage}{0.6\textwidth}
\centering
\includegraphics[width=\textwidth]{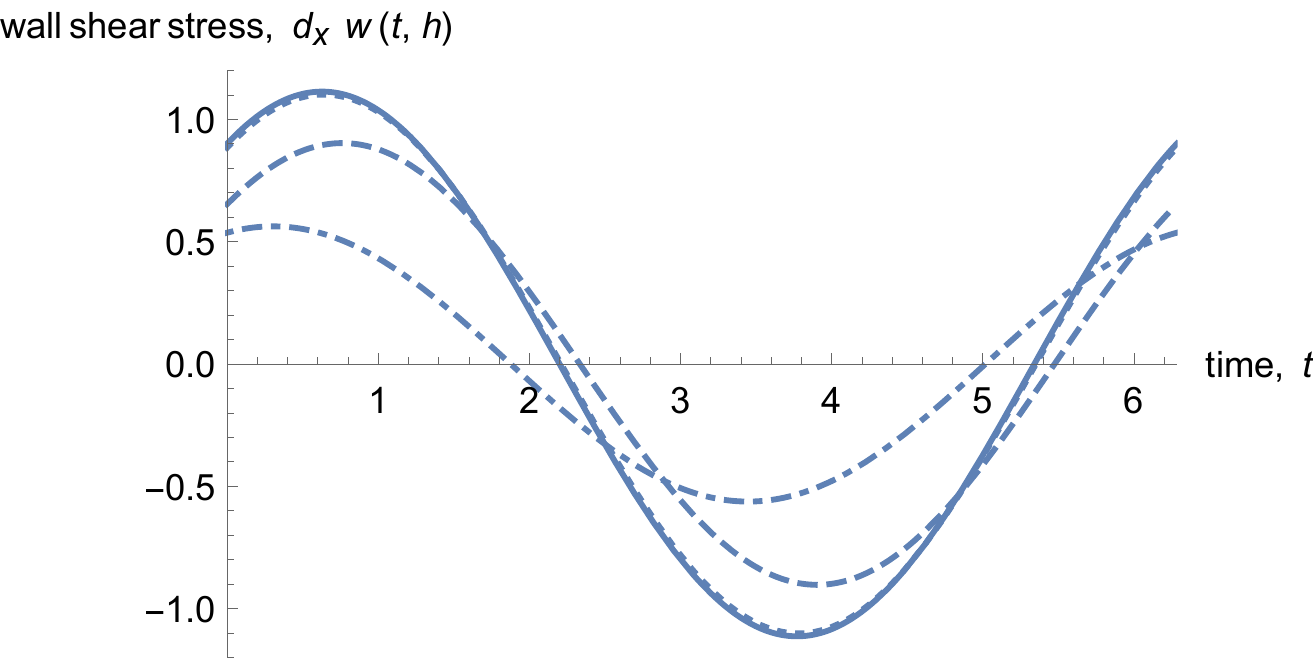}
\end{minipage}\quad
\begin{minipage}{0.35\textwidth}
\captionsetup{width=\linewidth}
\centering
\caption{Comparison of the wall shear stresses $\partial_x w(t,h)$ for fixed value of $\beta=1$ and three different values of $\alpha \in \{0.1, 4.2, 100\}$. For  the dot-dashed style $\alpha = 0.1$; for the dashed style $\alpha = 4.2$ and for the dotted style $\alpha = 100$. The full line corresponds to the Dirichlet solution, which is independent of $\alpha$ and $\beta$.}
\label{fig:compareperi2}
\end{minipage}
\end{figure}

In Figure~\ref{fig:compareperi} and Figure~\ref{fig:compareperi2}, we compare several examples of the periodic wall shear stress \eqref{xwth} of a solution corresponding to the dynamic slip condition with the reference wall shear stress \eqref{xbwth} of a solution that satisfies the Dirichlet condition. The constants are chosen as $h= \pi$ and $T=2\pi$. The eigenvalues of the reference ``Dirichlet" solution are the natural numbers (thanks to the proper choice of the parameter $h$).

In particular, we compare the values of the shear stresses on the boundary (and not the slip velocities as before), for the simple reason that due to the presence of the (complete) Dirichlet boundary condition, the slip velocity is equal to zero, and in this case the wall shear stress represents the behaviour on the boundary better.

In Figure~\ref{fig:compareperi}, we fix the value of $\alpha=1$ and compare the wall shear stresses for values $\beta \in \{0.1, 4.2, 100\}$ with the Dirichlet solution, which is independent of $\alpha$ and $\beta$.
On the other hand, in Figure~\ref{fig:compareperi2}, we did the opposite - we fixed the value of $\beta=1$ and compare the wall shear stress for values $\alpha \in \{0.1, 4.2, 100\}$ with the Dirichlet solution, which is independent of $\alpha$ and $\beta$.

Using these simulations, we can see how the values of $\alpha$ and $\beta$ influence the wall shear stresses of solutions. We basically see two effects - translation in time and  significant difference in the magnitudes of the solutions. In particular, we can observe that the dynamic solution narrows the Dirichlet solution in the case when the value of $\alpha$ in Figure~\ref{fig:compareperi2} and the value of $\beta$ in Figure~\ref{fig:compareperi} are large. To explain this, we notice that these parameters enter the formula~\eqref{xwth} via $A_i$ only, and when comparing~\eqref{condAi} and~\eqref{barui}, $A_i \to \bar{A}_i$ whenever $\alpha+\beta \to \infty$. Finally, both of them, $A_i$ in~\eqref{xwth} and $\bar{A}_i$ in~\eqref{xbwth}, are present in the second power.

\section{Function spaces}\label{spaces}

We work with a special type of boundary condition which includes the time derivative of the velocity of the fluid weighted by the parameter $\beta$. Such a structure demands a definition of specific function spaces, as well. In this part, we introduce the Gelfand triplets that consist of the function spaces which take into account our general boundary condition. 

\subsection*{Gelfand triplet}

For $\o$ a Lipschitz domain in $\R^d$, $\beta\geq 0$ and $r \in(1, \infty)$\footnote{In our result, Theorem~\ref{uNSrthm}, we only allow $r \in(6/5, \infty)$. However, this restriction is arising due to the lack of compactness in the convective term, i.e., it is initiated by the properties of the system. For other problems, e.g. the Stokes-like one, where the convective term is not present, we can use this theory for any $r \in(1, \infty)$.}, we define  $\mathcal{V} \subset \C^{0,1}(\overline{\o}) \times \C^{0,1}(\po)$ as
\begin{equation*}
\mathcal{V} := \{ (\vv, \g) \in \C^{0,1}(\overline{\o}) \! \times \! \C^{0,1}(\po); \diver \vv = 0\text{ in } \o, \vv \cdot \n = 0\text{ and } \vv = \g \text{ on } \po \}.
\end{equation*}
With the help of $\mathcal{V}$, we also define
\begin{align}
V_r &:= \overline{\mathcal{V}}^{\|\cdot\|_{V_r}},  \text{ where } \|(\vv, \g)\|_{V_r}:= \|\vv\|_{W^{1,r}(\o)} + \|\vv\|_{L^2(\o)}+ \|\g\|_{L^2(\po)} \label{V}\\
H &:= \overline{V_r}^{\|\cdot\|_H}, \text{ where } \|(\vv, \g)\|^2_H:= \|\vv\|^2_{L^2(\o)}+ \beta \|\g\|^2_{L^2(\po)}. \label{H}
\end{align}
For $r=2$, due to the Trace theorem, the norm on $V_2$ defined by~\eqref{V} is equivalent to the $W^{1,2}$-norm on $V_2$. Also, using the definitions of the $V_r$- and $H$-norms, we emphasize that for $(\vv, \g) \in V_r$ it holds that $\g = \tr \vv$ on $\po$, however, it need not be the case for $(\vv, \g) \in H$, since the latter one does not keep the Sobolev property of the function and the trace may not exists. Moreover, even if the trace of $\vv$ exists it is not necessarily equal to $\g$.

As $V_r$ is a closed subspace of $(\Wndr(\o)\cap L^2(\o))\times L^2(\po)$, which is a reflexive separable space, it is itself reflexive and separable. Also, $H$ is a Hilbert space identified with its own dual $H\equiv H^*$ with the inner product defined by
\begin{equation}
\left( (\tilde{\vv}, \tilde{\g}), (\vv, \g) \right)_{H}:= \int_{\o} \tilde{\vv} \cdot \vv \d x + \beta \int_{\po} \tilde{\g} \cdot \g \d S.
\end{equation}
By definition, $V_r$ is continuously embedded into $H$ and is also dense in $H$, therefore also the embedding $V_r \hookrightarrow H$ is  dense. Next, restricting every functional $f\in H^*$ to $V_r \subset H$, we get that $H^*\!\equiv\! H$ is embedded in $V_r^*$. This last embedding is also continuous because the adjoint map $i^*:H^*\to V^*$ to the continuous embedding $i:V\to H$ is continuous. Finally, the embedding $H^*\hookrightarrow V_r^*$ is dense because $V_r$ is reflexive and dense in $H$ (cf. \cite[Remark 17, p. 46]{Brezis}).
Thus, we have the Gelfand triplet
\begin{equation}\label{gelfand}
V_r \hookrightarrow  H\equiv H^*  \hookrightarrow V_r^*,
\end{equation}
and both embeddings are continuous and dense. Moreover, for $r>2d/(2+d)$, $W^{1,r}(\o)$ is compactly embedded into $L^2(\o)$ and for $r>2d/(d+1)$, the trace operator is compact from $W^{1,r}(\o) \to L^2(\po)$, according to the corollary of the Trace theorem. Therefore,
\begin{equation*}
V_r \hookrightarrow \hookrightarrow  H ~\text{ whenever }~ r>\frac{2d}{d+1}.
\end{equation*}

We define the duality pairing between $V_r$ and $V_r^*$ in a standard way as a continuous extension of the inner product $(\cdot, \cdot)_H$ on $H$. That is, for any $(\vv, \g) \in V_r \subset H$ and $(\tilde{\vv}, \tilde{\g})\in H^* \subset V_r^*$ we have
\begin{equation*}
\langle (\tilde{\vv}, \tilde{\g}), (\vv, \g) \rangle_{V_r} := \left( (\tilde{\vv}, \tilde{\g}), (\vv, \g) \right)_H.
\end{equation*}
Subsequently, for any $(\vv, \g) \in V_r$ and $(\tilde{\vv}, \tilde{\g})\in V_r^*$ we define
\begin{equation}\label{dualityV}
\langle (\tilde{\vv}, \tilde{\g}), (\vv, \g)  \rangle_{V_r} := \lim_{k\to +\infty} ((\tilde{\vv}^k, \tilde{\g}^k),(\vv, \g))_H,
\end{equation}
where $\{(\tilde{\vv}^k, \tilde{\g}^k)\}_{k\in \mathbb{N}}$ is a sequence in $H^*$ converging to $(\tilde{\vv}, \tilde{\g})$ in $V_r^*$.

Finally, we specify how to generate the duality pairing for object defined only inside of $Q$ or $\Omega$, which is for example the case of the external body forces $\f$. Hence, for $\f \in (W^{1,r}_{\n} (\o))^*$, we can identify it with $(\f,\0)\in V_r^*$, and we can write
\begin{equation}\label{dualityVf}
\langle \f, \vp \rangle_{V_r}:= \langle (\f,\0), (\vp, \tr \vp) \rangle_{V_r}= \lim_{k\to +\infty} \io \f^k \cdot \vp \d x = \langle \f, \vp \rangle_{W^{1,r}_{\n}(\o)},
\end{equation}
for any $(\vp, \tr \vp) \in V_r$, where $\{(\f^k, \0\}_{k\in \mathbb{N}}$ is a sequence in $H$ converging to $(\f, \0)$ in $V^*_r$. Note that in the case when $\f\in L^2(\o)$, this definition just means
\begin{equation}\label{fvL2}
\langle \f, \vp \rangle_{L^2(\o)}=\io \f\cdot \vp \d x,
\end{equation}
which is exactly the formula requiring the consistency of a definition of a weak solution. It is evident that this term does not see any information coming from the boundary $\po$. Although, it would not be the case if we considered the generalisation \eqref{genP5}. In this setting, for given $\g\in L^2(\po)$, we would set
\begin{equation}\label{MB}
\langle (\f,\g), (\vp, \tr \vp) \rangle_{V_r} := \langle \f, \vp \rangle_{W^{1,r}_{\n}(\o)} + \ipo \g\cdot \tr \vp \d S,
\end{equation}
which would again correspond to a proper definition of a weak solution. From another point of view, if we considered a standard couple $(\tilde{\vv}, \tilde{\g}) \in V_r^*$, this is a continuous linear functional which is bounded, i.e., for any $(\vp, \tr \vp) \in V_r$,
\begin{equation*}
\left| (\tilde{\vv}, \tilde{\g}) (\vp, \tr \vp) \right| \le \|(\tilde{\vv}, \tilde{\g})\|_{V^*_r} \left(\|\vp\|_{W^{1,r}(\o)} + \|\vp\|_{L^{2}(\o)} + \|\tr \vp\|_{L^2(\po)} \right).
\end{equation*}
In contrary, for $(\f, \0) \in V^*_r$ we only have that for any $(\vp,\tr \vp) \in V_r$,
\begin{equation*}
|(\f, \0)(\vp, \tr \vp)| \le \|(\f, \0)\|_{V^*_r} \left(\|\vp\|_{W^{1,r}(\o)} + \|\vp\|_{L^{2}(\o)} \right).
\end{equation*}

\noindent
\textit{Notation.}
\begin{itemize}
\item[$\quad$]
\item[1.] For simplicity, the fact that $(\vp, \tr \vp) \in V_r$ or $H$ will be only denoted by $\vp \in V_r$ or $H$, respectively, however, understood in the sense of definition of the corresponding space. Especially, we always write $\f \in V_r^*$ and understand it in the sense $(\f, \0)\in V^*_r$.
\item[2.] For $r=2$, we simplify the notation and denote the Hilbert space $V:= V_2$.
\item[3.] In the space $V$, the $V$-norm defined in \eqref{V}, the $W^{1,2}$-norm, and the norm generated by the scalar product defined in \eqref{scalarproduct} are equivalent. We use them interchangeably.
\end{itemize}

The basis orthogonal in $V$ and orthonormal in $H$ is defined and studied in the Appendix~\ref{basis}.

\section{Maximal monotone graphs}\label{graphs}

In this part, we only shortly introduce the theory for the maximal monotone $r$-graphs. As opposed to other works on the maximal monotone graph setting, here we do not assume the existence of a Borel measurable selection operator, i.e., a mapping $\S^*:\R^{d\times d} \to \R^{d \times d}$ fulfilling $(\S^*(\DD),\DD)\in \A$.

A detailed attention to the formulations and proofs of the lemmata from this section is paid in~\cite[Section~4]{BuMaMa} and similarly also in~\cite[Chapter 4]{EM}. We use the results from these works in order to approximate the graphs, and focus on the dynamic slip effects in the proof presented here.

Due to~\eqref{NSrf} and \eqref{NSrg}, the implicit character of the response of the fluid inside the domain is described via maximal monotone $r$-graph (and on the boundary via maximal monotone $2$-graph), so we start with its definition.

\begin{definition}[Maximal monotone $r$-graph]\label{maxmongrafA}
Let $r \in (1, \infty)$, $r' := r/(r-1)$ and let $\A \subset \R^{d \times d} \times \R^{d \times d}$. We say that $\A$ is a maximal monotone $r$-graph, if
\begin{itemize}
\item[(A1)] $(\0,\0) \in \A$,
\item[(A2)] monotonicity: for any $(\S_1,\DD_1), (\S_2,\DD_2) \in \A$,
\begin{equation*}
(\S_1 - \S_2):(\DD_1 - \DD_2) \ge 0,
\end{equation*}
\item[(A3)] maximality: if for some $(\S,\DD) \in \R^{d \times d} \times \R^{d \times d}$ and all $(\oS,\oD) \in \A$
\begin{equation*}
(\S - \oS):(\DD - \oD) \ge 0
\end{equation*}
holds, then $(\S,\DD) \in \A$,
\item[(A4)] $r$-coercivity: there exist $C_1, C_2 \in \R_+$ such that for all $(\S,\DD) \in \A$ there holds
\begin{equation*}
\S :\DD \ge C_1(|\S|^{r'} +|\DD|^{r})-C_2.
\end{equation*}
\end{itemize}
\end{definition}

In general, the maximal monotone graph theory can cover a rather wide class of models, especially the implicit ones. However, the explicit description allows us to approach the problem via proper Galerkin approximation, in particular, when estimating the time derivative on the Galerkin level, where the properties of the explicit basis\footnote{The key property is the continuity of the projection operator in the space $V$ \eqref{vlV}. We believe that for smooth domains, one could construct a basis for which the projection would be continuous also in $V_r$, however, we omit such procedure here in order to avoid the technical difficulties.} and the continuity of the selection are heavily used. Therefore, our primary goal is to approximate maximal monotone $r$-graph by an explicit maximal monotone $2$-graph.

In what follows, we construct the approximative graphs $\Ae$ and $\Aee$, which bring us to the situation of a $2$-graph with explicit formulation. We postulate a lemma about this approximative property, as well as several other lemmata that describe the properties of the maximal monotone graphs. Finally, we make a note on the graph $\B$ which describes the constitutive relation on the boundary $\Gamma$.

\noindent
\textit{Construction.} Let $\A$ be a maximal monotone $r$-graph and let $\e>0$. We define
\begin{subequations}\label{AAA}
\begin{align}
\Ae &:= \{(\tS,\tD) \in \R^{d \times d} \times \R^{d \times d};\, \exists (\oS,\oD) \in \A, \tS = \oS, \tD = \oD + \e \oS\}, \label{Ae}\\
\Aee &:= \{(\S,\DD) \in \R^{d \times d} \times \R^{d \times d};\, \exists (\tS,\tD) \in \Ae, \S = \tS + \e \tD, \DD = \tD\}. \label{Aee}
\end{align}
\end{subequations}

\begin{lemma}\label{Agraf}
Let $\A$ be a maximal monotone $r$-graph. Then for every $\e \in (0,1)$, $\Aee$ is a maximal monotone $2$-graph. Moreover, there exists a unique $\S_\e^*: \R^{d\times d} \to \R^{d\times d}$, which is Lipschitz continuous and uniformly monotone, and satisfies
\begin{equation}\label{selection}
(\S, \DD) \in \Aee ~\Longleftrightarrow~ \S = \S_\e^*(\DD).
\end{equation}
Also, for an arbitrary measurable and bounded $U \subset Q$, let $\Se, \De : U \to \R^{d\times d}$ be such that $(\Se, \De) \in \Aee$ almost everywhere in $U$, and let
\begin{equation}\label{intbound}
\int_U \Se : \De \d x  \d t \leq C ~\text{ uniformly with respect to } \e.
\end{equation}
Then there exist $\S \in L^{r'}(U)$, $\DD \in L^r(U)$ so that for subsequences
\begin{equation}\label{min2r}
\begin{aligned}
\Se &\tow \S &&\text{ weakly in } L^{\min\{2,r'\}}(U), \\
\De &\tow \DD &&\text{ weakly in } L^{\min\{2,r\}}(U).
\end{aligned}
\end{equation}
Moreover, if
\begin{equation}\label{assomez}
\limsup_{\e \to 0_+} \int_U \Se : \De \d x  \d t \leq \int_U \S : \DD \d x \d t,
\end{equation}
then $(\S, \DD) \in \A$ almost everywhere in $U$ and for a subsequence,
\begin{equation}\label{limitanasobku}
(\Se : \De) \tow (\S:\DD)~ \text{ weakly in }~L^1(U).
\end{equation}
\end{lemma}
The above lemma deals with the convergence of approximative graphs and corresponding approximative quantities. For the sake of completeness, we also formulate the result about the fixed graph and the convergence properties therein.
\begin{lemma}\label{aproxvA}
For every $n \in \N$, let $(\S^n,\DD^n)\in \A$, and let
\begin{equation*}
\int_U \S^n : \DD^n \d x \d t \leq C ~\text{ uniformly with respect to}~ n
\end{equation*}
for $U \subset Q$ measurable and bounded. Then
\begin{align*}
\S^n &\tow \S &&\text{weakly in } L^{r'}(U), \\
\DD^n &\tow \DD &&\text{weakly in } L^{r}(U).
\end{align*}
Moreover, if
\begin{equation*}
\limsup_{n \to \infty} \int_U \S^n : \DD^n \d x \d t \leq \int_U \S : \DD \d x \d t,
\end{equation*}
then $(\S,\DD) \in \A$ almost everywhere in $U$ and $\S^n:\DD^n \tow \S:\DD$ weakly in $L^1(U)$.
\end{lemma}

The next lemma is in fact a replacement of the assumption about the existence of a measurable selection. Indeed, assuming that the Borel measurable selection exists, the claim of the following lemma is straightforward. In our case, we need to show that for any measurable $\DD$ there exists the corresponding measurable $\S$ such that $(\S,\DD)\in \A$.

\begin{lemma}\label{pidilema}
Let $r\in(1,\infty)$. For every $\DD \in L^r(Q)$ there exists $\S \in L^{r'}(Q)$ such that $(\S, \DD) \in \A$ almost everywhere in $Q$.
\end{lemma}

\begin{lemma}\label{lemaodhmin}
Let $(\Se,\De)\in \Aee$, then there exist $\tilde{C_1}, \tilde{C_2} >0$ such that
\begin{equation}\label{odhadAr2}
\Se : \De \geq \tilde{C_1}(|\Se|^{\min\{r', 2\}}) + |\De|^{\min\{r, 2\}}) - \tilde{C_2}.
\end{equation}
\end{lemma}

For clarity, we mention also the graph $\B$ acting on the boundary, and we formulate the convergence lemma for it. Also, we only formulate that the statements of other lemmata hold equivalently for $\B$, as well.

The maximal monotone $r$-graph $\B \subset \R^d \times \R^d$ and its approximations $\Be$ and $\Bee$ are defined to possess exactly the same qualities like their counterparts $\A$ from the Definition \ref{maxmongrafA}, $\Ae$ and $\Aee$ from \eqref{AAA}, although, being the subsets of $\R^d \times \R^d$. They are related via the approximating property described in the Lemma \ref{Bgraf}.

\begin{lemma}\label{Bgraf}
Let $\B \subset \R^d \times \R^d$ be a maximal monotone $r$-graph. Then for every $\e \in (0,1)$, $\Bee$ is a maximal monotone $2$-graph. Moreover, there exists a unique $\sg_\e^*: \R^d \to \R^d$, which is Lipschitz continuous and uniformly monotone, and satisfies
\begin{equation*}
(\sg, \vv) \in \Bee ~\Longleftrightarrow~ \sg = \sg_\e^*(\vv).
\end{equation*}
Also, for an arbitrary measurable and bounded $U \subset \Gamma$, let $\se, \ve : U \to \R^d$ be such that $(\se, \ve) \in \Bee$ almost everywhere in $U$ and let
\begin{equation}\label{intboundB}
\int_U \se \cdot \ve \d S \d t \leq C ~\text{ uniformly with respect to } \e.
\end{equation}
Then there exist $\sg \in L^{r'}(U)$, $\vv \in L^r(U)$ so that for subsequences
\begin{equation*}
\begin{aligned}
\se &\tow \sg &&\text{ weakly in } L^{\min\{2,r'\}}(U), \\
\ve &\tow \vv &&\text{ weakly in } L^{\min\{2,r\}}(U).
\end{aligned}
\end{equation*}
Moreover, if
\begin{equation}\label{assomezB}
\limsup_{\e \to 0_+} \int_U \se \cdot \ve \d S \d t \leq \int_U \sg \cdot \vv \d S \d t,
\end{equation}
then $(\sg, \vv) \in \B$ almost everywhere in $U$ and for a subsequence,
\begin{equation*}
(\se \cdot \ve) \tow (\sg\cdot\vv)~ \text{ weakly in }~L^1(U).
\end{equation*}
\end{lemma}

\section{Navier--Stokes-like flow}\label{NSr}

We finally apply the theory which was built and motivated up to now. We consider $\A$ to be a general $r$-graph with $r \in (6/5,\infty)$ and $\B$ to be a $2$-graph. This restriction for $r$ comes from the system and the condition on $\B$ is posited for practical reasons - the proof would work for any $\A$, $\B$ which are $r$- and $q$-graphs with appropriately redefined spaces $V_r$ to $V_r^q$. Also, we prove the result for $\beta>0$, as the case $\beta=0$ is already treated in the work~\cite{BGMS2}.

We recall the problem~\eqref{problem} and properly formulate the main result of this work. We start with the definition of a weak solution.

\begin{definition}\label{uNSrdef}
Let $T>0$, $\alpha, \beta> 0$, $r \in (6/5,\infty)$, $\o \subset \R^3$ be a Lipschitz domain, $\f\in L^{r'}(0,T; (V_r)^*)$, and $\vv_0 \in H$. Let $\A$ be a maximal monotone $r$-graph in $Q$ and  $\B$ be a maximal monotone $2$-graph on the boundary $\Gamma$. Set $z:= \max\{r, 5r/(5r-6)\}$. We say that the triplet $(\vv, \S, \sg)$ is a weak solution to the Navier--Stokes-like problem \eqref{problem} if
\begin{align*}
\vv &\in  L^r(0, T; V_r) \cap  \C_w([0, T]; H), \\
\pt \vv &\in L^{z'}(0, T; (V_z)^*),\\
\S &\in L^{r'}(Q),\\
\sg &\in L^\infty(0,T; L^2(\po)),
\end{align*}
the balance of linear momentum is satisfied in the weak sense, i.e., for almost all $t\in (0,T)$ and for all $\vp\in V_z$,
\begin{subequations}
\begin{equation}\label{WFNSr}
\langle \pt \vv,  \vp \rangle_{V_z} -\io (\vv \otimes \vv) : \nabla \vp \d x + \io \S : \DD\vp \d x + \alpha \ipo \sg \cdot \vp \d S  = \langle \f,  \vp \rangle_{V_z},
\end{equation}
and $(\S, \D) \in \A$ almost everywhere in~$Q$ and $(\sg, \vv) \in \B$ almost everywhere on~$\Gamma$. The initial condition is attained in the strong sense,
\begin{equation*}
\lim_{t\to 0_+} \|\vv(t) - \vv_0\|_H = 0.
\end{equation*}
Moreover, we say that a solution satisfies the energy inequality if for all $t \in (0,T)$,
\begin{equation}\label{enineqr}
\frac{1}{2}\|\vv(t)\|^2_H + \int_0^t \io \S:\D \d x \d \tau + \alpha \int_0^t \ipo \sg \cdot \vv \d S \d \tau \leq \it \langle \f,\vv \rangle_V \d \tau + \frac{1}{2}\|\vv_0\|^2_H.
\end{equation}
\end{subequations}
\end{definition}
First, we show that the above definition is compatible with the concept of a classical solution. Indeed, let us assume for a moment that the weak solution has an additional regularity
$$
\begin{aligned}
\partial_t \vv&\in L^2(0,T; H),\\
\diver \S &\in L^2(Q).
\end{aligned}
$$
and that also $\f \in L^2(Q)$. Then we can use integration by parts in~\eqref{WFNSr} and also the definition of a duality pairing in $V$ to obtain that for almost all $t\in (0,T)$ there holds
\begin{equation}\label{a-clas}
\io (\pt \vv + \diver (\vv \otimes \vv) -\diver \S)\cdot  \vp \d x +   \ipo (\pt \vv +\alpha \sg  + \S\n)\cdot \vp \d S  = \io \f \cdot  \vp \d x.
\end{equation}
In particular, \eqref{a-clas} holds for any smooth compactly supported $\vp$ having zero divergence and therefore we can use the de~Rham theorem to find a pressure $p$ such that \eqref{NSra} holds almost everywhere in $Q$. Furthermore, since the tangential part of $\vp$ can be arbitrary, it also follows from \eqref{a-clas} that \eqref{NSrd} is satisfied almost everywhere on $\Gamma$. Hence, the required compatibility condition holds true. It is worth noticing here, that in case we would consider the Stokes-like problem, i.e., the problem without the convective term, the required regularity $\partial_t \vv \in L^2(0,T; H)$ can be proven easily provided that the initial data $\vv \in V$ and the graph $\mathcal{A}$ represents a sub-differential of some convex potential.

\begin{theorem}\label{uNSrthm}
Let all assumptions of the Definition \ref{uNSrdef} be met. Then for any $T>0$, $\alpha, \beta>0$ and $r \in (6/5,\infty)$, there exists a weak solution to the system \eqref{problem}, which satisfies the energy inequality. Moreover, if $r \geq 11/5$, then the energy inequality holds with the equality sign.
\end{theorem}

To prove this result, we approach the problem~\eqref{problem} by a proper approximation. In order to define it, we introduce an auxiliary function $\Phi: \R \to \R$,
\begin{equation*}
\Phi(s):=
\begin{cases}
1 &\text{ if } |s| \in [0,1), \\
2-s &\text{ if } |s| \in [1,2), \\
0 &\text{ if } |s| \in [2,\infty),
\end{cases}
\end{equation*}
and for every $\delta \in (0,1)$, we define the cut-off function $\Phi_\delta:\R \to \R$,
\begin{equation}\label{cutoff}
\Phi_\delta(s):= \Phi(\delta s), ~\text{ therefore }~ \Phi_\delta(s)\to 1 ~\text{ as }~ \delta \to 0_+.
\end{equation}
This function helps us with splitting the approximation into two steps - in the first one, we converge in the graphs (i.e., with $\e$) with the cut-off convective term, and in the second one, we converge with~$\delta$ in order to obtain the result for the regular Navier--Stokes-like problem. But first, we prove the existence of a solution to the $\e, \delta$-approximating problem with cut--off convective term and continuous $2$-graphs $\Aee$ and $\Bee$.

Since we deal with a completely new setting of function spaces, we want to reprove all classical results in this new setting rigorously. Therefore, we also focus on attainment of initial condition and the validity of the energy inequality in detail. For such purposes, we define a certain function that is used frequently in what follows. For given arbitrary $T>0$,  $0< \kappa \ll 1$ and $t \in (0,T-\kappa)$, we consider $\eta \in \C^{0,1}([0,T])$ as a piece-wise linear function of three parameters, such that
\begin{equation}\label{eta}
\eta(\tau) :=
\begin{cases}
1 &\text{if } \tau \in [0,t),   \\
1 + \frac{t-\tau}{\kappa} &\text{if } \tau \in [t,t+\kappa),\\
0 &\text{if } \tau \in [t+\kappa, T].
\end{cases}
\end{equation}
This function is typically used in proofs on attainment of the initial data and on identification of the graphs.

On the other hand, we do not discuss in detail the standard methods and estimates related to Navier-Stokes-like systems and refer rather to \cite{BGMS2,DiRuWo,Lady3} for details.

\subsection{Existence for the \texorpdfstring{$\e,\delta$}{e,d}-approximating problem}

First, we prove existence of a solution $(\ve, \Se, \se)$ for every $\e \in (0,1)$ and for every $\delta \in (0,1)$ to the problem
\begin{subequations}\label{NS2problem}
\begin{align}
\diver \ve &=0 &&\text{in } Q,\label{NS2b}\\
\pt \ve + \diver\left((\ve \otimes \ve) \Phi_\delta(|\ve|^2)\right)- \diver \Se + \nabla p &= \f &&\text{in } Q,\label{NS2a} \\
-(\Se \n)_\tau &= \alpha \se + \beta \pt \ve &&\text{on } \Gamma,\label{NS2c}\\
\ve\cdot\n&=0 &&\text{on } \Gamma,\label{NS2d}\\
\ve(0)&=\vv_0  &&\text{in } \overline{\o},\label{NS2e}\\
(\Se,\Dve) &\in \Aee &&\text{in } Q,\label{NS2f}\\
(\se,\ve) &\in \Bee &&\text{on } \Gamma,\label{NS2g}
\end{align}
\end{subequations}
where $\Aee$ and $\Bee$ are constructed from $\A$ and $\B$, respectively, according to~\eqref{Aee}, and they are $2$-graphs with selection due to Lemmata \ref{Agraf} and \ref{Bgraf}.

For simplicity, we drop using the index $\e$ and from now on, we look for $(\vv, \S, \sg)$ instead of $(\ve, \Se, \se)$, however, we continue writing $\Aee$ and $\Bee$ to enhance the use of the approximating graphs.

We know that for $\Aee$, there exists $\S^*:\R^{d \times d} \to \R^{d \times d}$ such that
\begin{equation}\label{selectAN}
(\S, \D) \in \Aee ~\Longleftrightarrow~ \S = \S^*(\D).
\end{equation}
Similarly, for $\Bee$, we denote the selection $\sg^*: \R^d \to \R^d$ and
\begin{equation}\label{selectBN}
(\sg, \vv) \in \Bee ~\Longleftrightarrow~ \sg = \sg^*(\vv).
\end{equation}
Moreover, both $\S^*$ and $\sg^*$ are Lipschitz continuous and uniformly monotone.

\begin{theorem}\label{NS2thm}
Let $T>0$, $\alpha, \beta > 0$, $\delta \in (0,1)$, $\o \subset \R^3$ be Lipschitz, $\f\in L^2(0,T;V^*)$ and $\vv_0 \in H$. Let $\Phi_\delta$ be defined by \eqref{cutoff}. Then there exists a triplet $(\vv, \S, \sg)$ such that
\begin{align*}
\vv &\in  L^2(0, T; V) \cap  \C([0, T]; H), \\
\pt \vv &\in L^2(0, T; V^*),\\
\S &\in L^2(Q),\\
\sg &\in L^2(\Gamma),
\end{align*}
and~\eqref{NS2problem} is satisfied in the weak sense, i.e., for almost all $t\in (0,T)$ and for all $\vp\in V$,
\begin{subequations}
\begin{equation}\label{WFNS2}
\langle \pt \vv,  \vp \rangle_V -\io (\vv \otimes \vv)\Phi_\delta(|\vv|^2) : \nabla \vp \d x + \io \S : \DD\vp \d x + \alpha \ipo \sg \cdot \vp \d S  = \langle \f,  \vp \rangle_V,
\end{equation}
and $(\S, \D) \in \Aee$ almost everywhere in~$Q$, and $(\sg, \vv) \in \Bee$ almost everywhere on~$\Gamma$. The initial condition is attained in the strong sense,
\begin{equation}\label{ICNS2}
\lim_{t\to 0_+} \|\vv(t) - \vv_0\|_H = 0.
\end{equation}
\end{subequations}
\end{theorem}

\subsubsection*{Proof of Theorem \ref{NS2thm}}

Let $\{\w_i\}_{i\in \N}$ be a basis of $V$ constructed in Appendix~\ref{basis}. Recall the definition of the selections \eqref{selectAN} and \eqref{selectBN} for $\Aee$ and $\Bee$, respectively. For every $n\in \mathbb{N}$, we define the Galerkin approximation
\begin{equation}\label{gaN2}
\vv^n(t,\x):= \sum_{i=1}^{n} c^n_i(t) \w_i(\x)~\text{ for }~(t, \x) \in Q,
\end{equation}
where the functions $c^n_i(t)$ are defined such that for $i=1,\dots,n$, they solve the following system of ordinary differential equations
\begin{subequations}
\begin{equation}\label{galerkinN2}
\begin{aligned}
(\pt \vv^n, \w_i)_H  &-\io (\vv^n \otimes \vv^n)\Phi_\delta(|\vv^n|^2):\nabla \w_i \d x+ \io\! \S^*(\D^n) :\DD\w_i \d x \\
&+ \alpha \ipo \sg^*(\vv^n) \cdot \w_i \d S =\langle \f, \w_i \rangle_V
\end{aligned}
\end{equation}
with initial conditions
\begin{equation}\label{galerkinICN2}
c^n_i( 0)= \io \vv_0 \cdot \w_i \d x + \beta \ipo \vv_0 \cdot \w_i \d S =(\vv_0,\w_i)_H.
\end{equation}
\end{subequations}
Due to the Carath\'eodory theory (recall that the selections are Lipschitz continuous), existence of such a solution is obtained in an interval $[0, t^n)$ for some $t^n \in (0,T)$ and thanks to the uniform estimates derived in the following part we can set $t^n=T$. Furthermore, recall the definition of the projection $P^n$ \eqref{proj} to see (using \eqref{gaN2} and \eqref{galerkinICN2}) that
\begin{equation}\label{odhpocpod}
\|\vv^n(0)\|^2_H = \| \sum_{i=1}^n (\vv_0,\w_i)_H \w_i\|^2_H = \|P^n \vv_0\|^2_H \leq \|\vv_0\|^2_H,
\end{equation}
where we used the estimate \eqref{vlH}, as $\{\w_i\}_{i\in \N}$ is an orthonormal basis of $H$. Moreover, from \eqref{proj} and \eqref{gaN2} we get for any $\vp \in V$ that
\begin{equation}\label{PffH}
(\vv^n, P^n\vp)_H = \left(\sum_{j=1}^n c^n_j \w_j, \sum_{i=1}^n (\vp, \w_i)_H \w_i\right)_H = \left(\sum_{j=1}^n c^n_j \w_j, \vp \right)_H = (\vv^n, \vp)_H.
\end{equation}
The formula \eqref{galerkinN2} holds only for $\w$ from the linear hull of $\{\w_j\}_{j=1}^{n}$. However, with the help of~\eqref{PffH}, we can work with \eqref{galerkinN2} for any $\w \in V$.

\subsubsection*{Uniform estimates}

We multiply the $i$-th equation in \eqref{galerkinN2} by $c^n_i(t)$ and sum them together over $i=1,\dots,n$ to obtain\footnote{The convective term vanishes due to the fact that $\diver \vv^n =0$ in $Q$ and $\vv^n \cdot \n =0$ on $\Gamma$ as it follows from the following computation (for $P$ a primitive function to $\Phi_{\delta}$)
\begin{equation}\label{vorezcanc}
\begin{aligned}
\io &(\vv\otimes \vv) \Phi_{\delta}(|\vv|^2) : \nabla \vv \d x =\frac12 \io \vv\cdot \nabla |\vv|^2 \Phi_{\delta}(|\vv|^2)\d x\\
&=\frac12 \io \vv\cdot \nabla P(|\vv|^2) \d x=-\frac12 \io P(|\vv|^2)\diver \vv \d x=0.
\end{aligned}
\end{equation}
}
\begin{align}\label{testvnN2}
\frac12 \dt \|\vv^n\|^2_H + \io \S^*(\D^n) : \D^n \d x &+ \alpha\ipo  \sg^*(\vv^n) \cdot \vv^n \d S = \langle \f,  \vv^n \rangle_V,
\end{align}
Since, $\Aee$ and $\Bee$ are $2$-graphs and $(\S^*(\D^n), \D^n)\in \Aee$ and $(\sg^*(\vv^n), \vv^n)\in \Bee$, we can use the coercivity assumption $(A4)$, the Young and the Korn inequalities and the estimate \eqref{odhpocpod}, and it follows from \eqref{testvnN2} that there is a constant $C$ depending only on $\f$, $\vv_0$, $\Omega$ and $\varepsilon$ such that
\begin{equation}\label{uniformvnN2}
\begin{aligned}
\|\vv^n\|_{L^2(0,T;V) \cap L^\infty(0,T;H)} \leq C ~\text{ uniformly with respect to } n, \\
\|\S^*(\D^n)\|_{L^2(Q)} + \|\sg^*(\vv^n)\|_{L^2(\Gamma)} \leq C ~\text{ uniformly with respect to } n.
\end{aligned}
\end{equation}

Using the properties of the projection $P^n$, see \eqref{PffH}, we  can reconstruct the estimate for the time derivative for $\pt \vv^n$. We know that $\pt \vv^n \in H$ and using this information we show that it is also uniformly bounded in $L^{2}(0,T;V^*)$. For an arbitrary $\vp \in V$, using \eqref{PffH} and the continuity of the projection in $V$ \eqref{vlV},
\begin{align*}
\langle \pt \vv^n, \vp \rangle_V &= (\pt \vv^n, P^n\vp)_H = \io \pt \vv^n \cdot (P^n\vp) \d x + \beta \ipo \pt \vv^n \cdot (P^n\vp) \d S \\
&=\io\left((\vv^n\otimes\vv^n)\Phi_\delta(|\vv^n|^2) -\S^*(\D^n) \right)\!:\! \nabla (P^n\vp) \d x \\
&\quad - \alpha\ipo \sg^*(\vv^n) \!\cdot\! (P^n\vp) \d S  + \langle \f, P^n\vp \rangle_V \\
&\leq C \left( \|\S^*(\D^n)\|_{L^2(\o)} + \|\sg^*(\vv^n)\|_{L^2(\po)}+ C(\delta) + \|\f\|_{V^*}\right)\|\vp\|_V.
\end{align*}
Thus, we define $\mn := \{\vp \in V, \|\vp\|_V \leq 1\}$ and recall  \eqref{uniformvnN2}, and the assumption on $\f$ to obtain
\begin{equation}\label{uniformptvnN2}
\begin{split}
\iT \|\pt \vv^n\|^2_{V^*} \d t &=\iT \left(\sup_{\vp \in \mn} \langle \pt \vv^n, \vp\rangle _{V}\right)^2 \d t \\
 &\le C\iT  \left( \|\S^*(\D^n)\|_{L^2(\o)} + \|\sg^*(\vv^n)\|_{L^2(\po)}+ C(\delta) + \|\f\|_{V^*}\right)^2  \d t\\
 &\leq C ~\text{ uniformly with respect to } n.
\end{split}
\end{equation}

\subsubsection*{Limit passage}\label{NS2limit}

By virtue of the uniform estimates \eqref{uniformvnN2} and  \eqref{uniformptvnN2}, reflexivity of spaces~$V$ and~$V^*$, the Aubin--Lions lemma (recall the compact embedding $V\hookrightarrow\hookrightarrow H$) and integration by parts for Sobolev Bochner functions, there exist (not relabelled) subsequences and functions $\vv$, $\S$ and $\sg$ such that as $n \to \infty$,
\begin{subequations}\label{convergencesN2}
\begin{align}
\vv^n &\tow^* \vv &&\text{weakly$^*$ in } L^\infty(0, T; H),\\
\vv^n &\tow \vv &&\text{weakly in } L^2(0, T; V), \label{weakVNS2}\\
\pt \vv^n &\tow \pt \vv &&\text{weakly in } L^2(0,T;V^*), \\
\vv^n &\to \vv &&\text{strongly in } L^2(0, T; H), \label{convergence-strongNS2} \\
(\vv^n \otimes \vv^n)\Phi_\delta(|\vv^n|^2) &\to (\vv \otimes \vv)\Phi_\delta(|\vv|^2) &&\text{strongly in } L^\gamma(Q), \gamma \in [1, \infty)\label{convergenceLg2}, \\
\S^*(\D^n) &\tow \S &&\text{weakly in } L^2(Q), \\
\sg^*(\vv^n) &\tow \sg &&\text{weakly in } L^2(\Gamma).
\end{align}
\end{subequations}
We add a short comment on~\eqref{convergenceLg2}. For fixed $\delta$, $\vv \mapsto (\vv \otimes \vv)\Phi_\delta(|\vv|^2)$ is bounded and continuous, and this together with~\eqref{convergence-strongNS2} imply the almost everywhere convergence of $\vv^n$ to $\vv$ in $Q$. Then~\eqref{convergenceLg2} holds, and the result follows e.g. by the use of the Lebesgue dominated convergence theorem.

In \eqref{galerkinN2}, for any $\psi \in C^1(0, T)$ and $\vp \in V$, we multiply the $i$-th equation by $\psi (\vp, \w_i)_H$, sum over $i=1,\ldots, k$ for $k\leq n$ and integrate over $t\in(0,T)$ to get for every $k=1,\ldots, n$
\begin{align*}
\iT &\left(\pt \vv^n, P^k\vp \right)_H \psi \d t -\iq (\vv^n \otimes \vv^n)\Phi_\delta(|\vv^n|^2) : \nabla (P^k\vp) \psi \d x \d t \\
&+ \!\iq \S^*(\D^n)\!:\!\DD(P^k\vp)\psi \d x \d t + \alpha\!\ig \sg^*(\vv^n) \!\cdot\! (P^k\vp) \psi\d S \d t = \!\iT\! \langle \f, P^k\vp \rangle_V \psi \d t.
\end{align*}
Using the convergence results~\eqref{convergencesN2}, we can proceed with the limit $n \to \infty$. The limit integral holds for any~$\psi$, therefore we obtain
\begin{equation*}
\begin{aligned}
\langle \pt \vv, P^k\vp \rangle_V &-\io (\vv \otimes \vv)\Phi_\delta(|\vv|^2):\nabla (P^k\vp) \d x+ \io \S:\DD(P^k\vp) \d x \\
&+\alpha \ipo \sg \cdot (P^k\vp) \d S = \langle \f, P^k\vp \rangle_V
\end{aligned}
\end{equation*}
for almost all $t\in (0,T)$ and for all $k \in \N$. Finally, we can use the property of the projection $P^k \vp \to \vp$ in $V$ as $k \to \infty$ from \eqref{vlproj} and obtain the weak formulation~\eqref{WFNS2}.

\subsubsection*{Initial data attainment}

Since the initial condition involves also behavior on the boundary, we prove the attainment rigorously here, although it somehow follows step by step the standard setting with the only change in the definition of the function spaces. From the previous parts, we know that $\vv \in L^2(0,T; V)$ and $\pt \vv \in L^2(0,T; V^*)$, which implies that $\vv \in \C([0,T]; H)$. From the definition of the space $\C([0,T]; H)$, we get that
\begin{equation}\label{strongconvpocpod}
\vv(t) \to \vv(0) ~\text{ strongly in } H \text{ as } t \to 0_+.
\end{equation}
In what follows, we show that $\vv(t) \tow \vv_0$ weakly in $H$ as $t \to 0_+$, and these convergence results together identify the limit \eqref{ICNS2}, that we want to prove.

Let $0< \kappa \ll 1$ and $t \in (0,T-\kappa)$. We recall the definition of an auxiliary~$\eta$ in~\eqref{eta}, multiply~\eqref{galerkinN2} by this~$\eta$, and integrate over~$(0,T)$ to obtain for every $i = 1, \ldots, n$
\begin{align*}
\iT (\pt \vv^n, \w_i)_H \eta \d \tau &+ \iq \left(\S^*(\D^n) - (\vv^n \otimes \vv^n)\Phi_\delta(|\vv^n|^2)\right):\nabla \w_i \eta\d x \d \tau \\
&+ \alpha \ig \sg^*(\vv^n) \cdot \w_i \eta \d S \d \tau  = \iT \langle \f, \w_i\rangle_V \eta \d \tau.
\end{align*}
Next, we integrate by parts in the first term, use that $\eta(T)=0$, and the equality in \eqref{odhpocpod} ($\vv^n(0) = P^n \vv_0$), to get
\begin{align*}
-\iT (\vv^n, \w_i)_H \eta' \d \tau &+ \iq \left(\S^*(\D^n) - (\vv^n \otimes \vv^n)\Phi_\delta(|\vv^n|^2)\right):\nabla \w_i \eta\d x \d \tau \\
&+ \alpha \ig \sg^*(\vv^n) \cdot \w_i \eta \d S \d \tau= \iT \langle \f, \w_i\rangle_V \eta \d \tau+ (P^n \vv_0, \w_i)_H \eta(0),
\end{align*}
and this equation is ready for the use of the weak convergence results \eqref{convergencesN2} and the convergence of the projection \eqref{vlconv} to obtain for any $i \in \N$ that
\begin{align*}
-\iT (\vv, \w_i)_H \eta' \d \tau &+ \iq \left(\S - (\vv \otimes \vv)\Phi_\delta(|\vv|^2)\right):\nabla \w_i \eta\d x \d \tau \\
&+ \alpha \ig \sg \cdot \w_i \eta \d S \d \tau = \iT \langle \f, \w_i\rangle_V \eta \d \tau+ (\vv_0, \w_i)_H \eta(0).
\end{align*}
Next, we use the properties of $\eta$, namely that $\eta(\tau) = 1$ for $\tau\in [0,t)$, $\eta(\tau) = 0$ for $\tau\in (t + \kappa, T]$, and $\eta'(\tau) = -\frac{1}{\kappa}$ for $\tau \in (t,t + \kappa)$. Then we have
\begin{align*}
\frac{1}{\kappa}\int_t^{t+\kappa} (\vv, \w_i)_H \d \tau &+ \int_{Q_{t+\kappa}} \left(\S - (\vv \otimes \vv)\Phi_\delta(|\vv|^2)\right):\nabla \w_i \eta \d x \d \tau \\
 &+ \alpha \int_{\Gamma_{t+\kappa}} \sg \cdot \w_i \eta\d S \d \tau =  \int_0^{t+\kappa} \langle \f, \w_i\rangle_V \eta\d \tau + (\vv_0, \w_i)_H.
\end{align*}
Further, we wish to proceed with the limit as $\kappa \to 0_+$. In the first term, the integrand is well-defined ($\vv \in \C([0,T];H)$), and the mean-value integral converges to $(\vv(t), \w_i)_H$. In the other terms, we take the limit as $\kappa \to 0_+$ together with $t \to 0_+$, use that all quantities are integrable in appropriate spaces and arrive at
\begin{align*}
\lim_{t \to 0_+}(\vv(t), \w_i)_H = (\vv_0, \w_i)_H.
\end{align*}
This holds for every $i \in \N$, and since $\{\w_i\}_{i\in \N}$ is a basis in $H$, this is nothing but the weak convergence result we hoped for, and it identifies the strong limit in \eqref{strongconvpocpod} of the initial condition in $H$.

\subsubsection*{Graphs identification}

After proceeding with the limit, it remains to show that the limiting objects relate to each other in the way we want them to, i.e., that $(\S,\D) \in \Aee$ and $(\sg,\vv)\in \Bee$. To do so, we multiply~\eqref{testvnN2} by piece-wise linear~$\eta(t)$ defined in~\eqref{eta} and integrate over~$(0,T)$ to obtain
\begin{align*}
\int_{Q_{t + \kappa}} &\S^*(\D^n) : \D^n \eta \d x \d \tau +\alpha \int_{\Gamma_{t + \kappa}} \sg^*(\vv^n) \cdot \vv^n \eta \d S \d \tau \\
&= \int_0^{t+\kappa}\langle \f, \vv^n \rangle_V \eta \d \tau + \frac12 \|P^n \vv_0\|_H^2 - \frac{1}{2\kappa}\int_t^{t + \kappa} (\vv^n, \vv^n)_H \d \tau.
\end{align*}
Since $\S^*(\0) = \0$ and it is monotone (and the same holds for $\sg^*$), we have for every $n \in \N$
\begin{align*}
\S^*(\D^n) : \D^n &\geq 0,\\
\sg^*(\vv^n) \cdot \vv^n &\geq 0.
\end{align*}
Therefore,
\begin{align*}
\limsup_{n \to \infty} &\int_{Q_{t}} \S^*(\D^n) : \D^n \d x \d \tau + \alpha\int_{\Gamma_{t}} \sg^*(\vv^n) \cdot \vv^n \d S \d \tau \\
&\leq \limsup_{n \to \infty} \int_{Q_{t + \kappa}} \S^*(\D^n) : \D^n \eta \d x \d \tau + \alpha\int_{\Gamma_{t + \kappa}} \sg^*(\vv^n) \cdot \vv^n \eta \d S \d \tau \\
&= \limsup_{n \to \infty} \int_0^{t+\kappa}\langle \f, \vv^n \rangle_V \eta \d \tau + \frac12 \|P^n \vv_0\|_H^2 - \liminf_{n \to \infty} \frac{1}{2\kappa}\int_t^{t +\kappa} (\vv^n, \vv^n)_H \d \tau \\
&\leq \int_0^{t+\kappa}\langle \f, \vv \rangle_V \eta \d \tau + \frac12 \|\vv_0\|_H^2 - \frac{1}{2\kappa}\int_t^{t + \kappa} (\vv, \vv)_H \d \tau,
\end{align*}
where we used the results from \eqref{convergencesN2} and the weak lower semicontinuity of the norm. If we proceed with $\kappa \to 0_+$, we note that the left hand side is independent of $\kappa$, and on the right hand side, all quantities are well-defined for such limit (since $\vv \in \C([0,T]; H)$), and using again the weak lower semicontinuity of the norm we finally obtain for an arbitrary $t\in(0,T)$
\begin{equation}\label{limitujuce}
\begin{aligned}
\limsup_{n \to \infty} &\int_{Q_{t}} \S^*(\D^n) : \D^n \d x \d \tau + \alpha\int_{\Gamma_{t}} \sg^*(\vv^n) \cdot \vv^n \d S \d \tau \\
&\leq \it\langle \f, \vv \rangle_V \d \tau + \frac12 \left( \|\vv_0\|_H^2 - \|\vv(t)\|_H^2 \right).
\end{aligned}
\end{equation}
Now, we set $\vp := \vv$ in \eqref{WFNS2}, use~\eqref{vorezcanc}, integrate over time $(0,t)$, and use that we can integrate by parts in the duality (thanks to the fact that we have the Gelfand triplet) and the attainment of the initial value,
\begin{equation}\label{limitne}
\begin{aligned}
\int_{Q_{t}} \S : \D \d x \d \tau &+\alpha \int_{\Gamma_{t}} \sg \cdot \vv \d S \d \tau = \it\langle \f, \vv \rangle_V - \langle \pt \vv, \vv \rangle_V \d \tau \\
&= \it\langle \f, \vv \rangle_V \d \tau + \frac12 \left( \|\vv_0\|_H^2 - \|\vv(t)\|_H^2 \right).
\end{aligned}
\end{equation}
If we compare \eqref{limitujuce} and \eqref{limitne}, we obtain the condition
\begin{equation}\label{odhadbox}
\begin{aligned}
\limsup_{n \to \infty} &\int_{Q_{t}} \S^*(\D^n) : \D^n \d x \d \tau + \alpha\int_{\Gamma_{t}} \sg^*(\vv^n) \cdot \vv^n \d S \d \tau \\
&\leq \int_{Q_{t}} \S : \D \d x \d \tau + \alpha\int_{\Gamma_{t}} \sg \cdot \vv \d S \d \tau.
\end{aligned}
\end{equation}
Now, let $\W \in L^2(Q)$ and $\w \in L^2(\Gamma)$ be arbitrary, then by monotonicity of the graphs
\begin{align*}
0&\leq \iqt (\S^*(\D^n)-\S^*(\W)):(\D^n-\W) \d x \d \tau \\
&\quad + \alpha\igt (\sg^*(\vv^n)-\sg^*(\w)) \cdot (\vv^n-\w) \d S \d \tau \\
&=\iqt \S^*(\D^n):\D^n  \d x \d \tau + \alpha\igt \sg^*(\vv^n) \cdot \vv^n \d S \d \tau \\
&\quad -\iqt \S^*(\D^n):\W + \S^*(\W):(\D^n-\W) \d x \d \tau \\
&\quad -\alpha\igt \sg^*(\vv^n) \cdot \w + \sg^*(\w) \cdot (\vv^n - \w)\d S \d \tau.
\end{align*}
For the first two integrals, we use the estimate~\eqref{odhadbox}, and for the rest, we use the weak convergence results in \eqref{convergencesN2},
\begin{equation*}
0\leq \iqt (\S-\S^*(\W)):(\D-\W) \d x \d \tau + \alpha\igt (\sg-\sg^*(\w)) \cdot (\vv-\w) \d S \d \tau.
\end{equation*}
Now, we set $\W:= \D \pm \mu \Z$, $\w:= \vv \pm \mu \z$, divide by $\mu>0$ and let $\mu \to 0_+$ (at this point we use the continuity of the selections) to obtain for arbitrary $\Z$ and $\z$ and given $\alpha \geq 0$
\begin{equation}\label{minty}
0= \iqt (\S-\S^*(\D)):\Z \d x \d \tau +\alpha \igt (\sg-\sg^*(\vv)) \cdot \z \d S \d \tau.
\end{equation}
Here, we followed the Minty method from \cite{Minty}, with small modifications in order to adapt it to our setting.

Finally,  setting $\Z:=(\S-\S^*(\D))$ and $\z:=(\sg-\sg^*(\vv))$ in~\eqref{minty} implies\footnote{In case $\alpha=0$, it does not imply that $\sg = \sg^*(\vv)$. However, in this case, we can use \eqref{convergence-strongNS2} to obtain the strong convergence $\vv^n \to \vv$ in $L^2(0,T; L^2(\partial \Omega))$ and due to continuity of $\sg^*$ the claim follows.} that $\S = \S^*(\D)$ in $Q_t$ and $\sg = \sg^*(\vv)$ in $\Gamma_t$ for any $t\in (0,T)$, therefore $(\S, \D) \in \Aee$ almost everywhere in $Q$, and $(\sg,\vv)\in\Bee$ almost everywhere in $\Gamma$.

\subsection{Limit \texorpdfstring{$\varepsilon \to 0_+$}{e}}

Having the existence of a solution $(\ve, \Se, \se)$ for every $\e \in (0,1)$ and for every $\delta \in (0,1)$ to the problem~\eqref{NS2problem}, the next step is to prove the existence of a solution to the same problem, however, now with $\A$ a maximal monotone $r$-graph, $r\in(6/5,\infty)$, and $\B$ a maximal monotone $2$-graph, possibly without a Borel measurable selection. However, this was done in~\cite{BuMaMa} for a general parabolic problem, and we do not repeat the whole procedure here rigorously but we just point out the essential steps.
%
%
Indeed, due to the presence of the cut-off function $\Phi_\delta$, the convective term can be understood as a compact perturbation and satisfies the strong convergence result \eqref{convergenceLg2}, and therefore creates no additional difficulties in the limit passage as $\e \to 0_+$. Hence, the goal of this section is to prove the following result.
\begin{theorem}\label{NSrethm}
Let $T>0$, $\alpha, \beta> 0$, $\delta \in (0,1)$, $r \in (6/5,\infty)$, $\o \subset \R^3$ be Lipschitz, $\f\in L^{r'}(0,T;(V_r)^*)$ and $\vv_0 \in H$. Then there exists a triplet $(\vd, \Sd, \sd)$ such that
\begin{align*}
\vd &\in  L^r(0, T; V_r) \cap  \C([0, T]; H), \\
\pt \vd &\in L^{r'}(0, T; V_r^*),\\
\Sd &\in L^{r'}(Q),\\
\sd &\in L^\infty(0,T; L^2(\po)),
\end{align*}
and for almost all $t\in (0,T)$  and for all $\vp\in V_r$,
\begin{subequations}
\begin{equation}\label{WFNSre}
\begin{aligned}
\langle \pt \vd,  \vp \rangle_{V_r} &-\io \left((\vd \otimes \vd) \Phi_\delta(|\vd|^2)\right) : \nabla \vp \d x \\
&+ \io \Sd : \DD\vp \d x + \alpha \ipo \sd \cdot \vp \d S = \langle \f,\vp \rangle_{V_r}
\end{aligned}
\end{equation}
and $(\Sd, \Dd) \in \A$ almost everywhere in $Q$, and $(\sd, \vd) \in \B$ almost everywhere on $\Gamma$. The initial condition is attained in the strong sense,
\begin{equation}\label{ICNSre}
\lim_{t\to 0_+} \|\vd(t) - \vv_0\|_H = 0.
\end{equation}
\end{subequations}
\end{theorem}
\subsubsection*{Sketch of the proof of Theorem~\ref{NSrethm}:}
We use Theorem~\ref{NS2thm} and for any $\varepsilon\in (0,1)$ we have the solution $(\vv^{\varepsilon}, \S^{\varepsilon}, \se)$ fulfilling \eqref{WFNS2}. Setting, $\vp:=\vv^{\varepsilon}$ in \eqref{WFNS2} and following the estimates done in preceding section, we obtain the starting inequality
\begin{equation}\label{addM}
\sup_{t\in(0,T)}\|\vv^{\varepsilon}(t)\|_H^2 + \iT \io \S^{\varepsilon}: \DD\vv^{\varepsilon} \d x \d t + \iT \ipo \se\cdot \vv^{\varepsilon} \d x \d t \le C~\text{ uniformly with respect to $\varepsilon$}.
\end{equation}
Next, since $(\S^{\varepsilon}, \DD \vv^{\varepsilon})\in \Aee$ and $(\se,\vv^{\varepsilon})\in \Bee$, we can use Lemma~\ref{Agraf} and Lemma~\ref{Bgraf} and thanks to \eqref{addM}, we have
\begin{equation}\label{addM2}
\begin{aligned}
\vv^{\varepsilon} &\tow^* \vv &&\text{weakly$^*$ in } L^\infty(0, T; H),\\
\vv^{\varepsilon} &\tow \vv &&\text{weakly in } L^{\min(r,2)}(0,T; V_{\min(r,2)}), \\
\S^{\varepsilon} &\tow \S &&\text{weakly in } L^{\min(2,r')}(Q), \\
\se &\tow \sg &&\text{weakly in } L^2(\Gamma).
\end{aligned}
\end{equation}
In addition, it also follows from Lemma~\ref{Agraf} and the Korn inequality that
\begin{equation}\label{spaceMB}
\S\in L^{r'}(Q) \quad \textrm{and} \quad \vv \in L^r(0,T; V_r).
\end{equation}
Then, following the computation in \eqref{uniformptvnN2} and using \eqref{addM2}, we also have
\begin{equation}\label{uniformptvepsN2}
\begin{split}
&\iT \|\pt \vv^{\varepsilon}\|^{\min(2,r')}_{V_{\max(2,r)}^*} \d t \\
&\le C\iT  \left( \|\S^{\varepsilon}\|_{L^{\min(2,r')}(\o)} + \|\se\|_{L^2(\po)}+ C(\delta) + \|\f\|_{V^*_{\max(2,r')}}\right)^{\min(2,r')}  \d t\\
 &\leq C ~\text{ uniformly with respect to } \varepsilon
\end{split}
\end{equation}
and consequently using also the Aubin--Lions lemma and the Trace theorem, we deduce
\begin{equation}\label{addM3}
\begin{aligned}
\pt\vv^{\varepsilon} &\tow \pt\vv &&\text{weakly in } L^{\min(r',2)}(0, T; V_{\max(r,2)}^*),\\
\vv^{\varepsilon} &\to \vv &&\text{strongly in } L^{1}(Q), \\
\vv^{\varepsilon} &\to \vv &&\text{strongly in } L^{1}(\Gamma).
\end{aligned}
\end{equation}
Having \eqref{addM2}, \eqref{uniformptvepsN2} and \eqref{addM3}, we can easily let $\varepsilon\to 0_+$ in \eqref{WFNS2} to obtain  \eqref{WFNSre} with one proviso, namely, that $\vp \in V_{\max(2,r)}$. However, thanks to \eqref{spaceMB}, we can improve the estimate for time derivative and conclude that
\begin{equation}\label{addM4}
\begin{aligned}
\pt\vv \in L^{r'}(0, T; V_{r}^*)
\end{aligned}
\end{equation}
and that \eqref{WFNSre} holds true for all $\vp\in V_r$. The attainment of the initial condition can be shown exactly as in the proof of Theorem~\ref{NS2thm}.

The crucial part is to check that $(\S, \DD\vv)\in \mathcal{A}$ and $(\sg,\vv)\in \mathcal{B}$. For this purpose, it is just enough to verify remaining assumptions of Lemma~\ref{Agraf}  and Lemma~\ref{Bgraf}, namely to show that
$$
\limsup_{\varepsilon \to 0_+} \int_Q \S^{\varepsilon} : \DD \vv^{\varepsilon} \d x \d t + \int_{\Gamma} \se\cdot \vv^{\varepsilon} \d S \d t\le \int_Q \S : \DD \vv \d x \d t + \int_{\Gamma} \sg\cdot \vv \d S \d t.
$$
This can be however achieved by repeating the procedure from the proof of Theorem~\ref{NS2thm}, namely, we set $\vp:=\vv^{\varepsilon}$ in the equation for $\vv^{\varepsilon}$ \eqref{WFNS2}, and we set $\vp:=\vv$ in the equation for $\vv$ \eqref{WFNSre}, let $\varepsilon\to 0_+$ and compare the limit. We do not provide more details here, since it is very similar to the preceding section and almost exactly the same as in \cite{BuMaMa}.

\subsection{Proof of Theorem \ref{uNSrthm}}

Having Theorem \ref{NSrethm} in hands, we proceed to the proof of Theorem~\ref{uNSrthm}, which in this situation means to explain the procedure of taking the limit as $\delta \to 0_+$ in~\eqref{WFNSre}.
We consider $r > 6/5$ and for every $\delta \in (0,1)$, we have $(\vd, \Sd, \sd)$, a solution according to Theorem~\ref{NSrethm} such that $(\Sd, \Dd) \in \A$ almost everywhere in $Q$ and $(\sd, \vd)\in \B$ almost everywhere on $\Gamma$.

\subsubsection*{Uniform estimates and limit passage}

To obtain a~priori estimates, we set $\vp := \vd$ in \eqref{WFNSre} (the term with the convective term cancels due to \eqref{vorezcanc}), integrate over time $(0,t)$, integrate by parts in the first term, use \eqref{odhpocpod} for the initial condition and ``usual" estimate for the duality on the right hand side with the help of the H\"older and Young inequalities, to obtain
\begin{equation*}
\sup_{t \in (0,T)}\|\vd(t)\|_H^2 + \iq \Sd \!:\! \Dd \d x \d t + \alpha \ig \sd \!\cdot\! \vd \d S \d t \leq C ~\text{uniformly with respect to } \delta.
\end{equation*}
Due to the $r$-coercivity of $\A$ and $2$-coercivity of $\B$, we obtain that
\begin{equation}\label{adMB7}
\begin{aligned}
\|\vd\|_{L^\infty(0,T;H)\cap L^r(0;T;V_r)\cap L^2(\Gamma)} &\leq C ~\text{uniformly with respect to } \delta,\\
\|\Sd\|_{L^{r'}(Q)} + \|\sd\|_{L^2(\Gamma)} &\leq C ~\text{uniformly with respect to } \delta.
\end{aligned}
\end{equation}
To improve the estimate of the terms $\sd, \vd$ on the boundary, we use that $\beta >0$ to obtain
\begin{equation*}
\vd \in L^\infty(0,T;H) \implies \vd \in L^\infty(0,T;L^2(\po)).
\end{equation*}
Then we can estimate
\begin{equation*}
C_1\left( |\sd|^2 + |\vd|^2 \right) - C_2 \leq \sd \cdot \vd \leq \frac{C_1^2}{2} |\sd|^2 + C |\vd|^2,
\end{equation*}
and subsequently
\begin{equation}\label{unifse}
\begin{aligned}
\sup_{t \in (0,T)} \|\sd(t)\|^2_{L^2(\po)} &\leq C \! \sup_{t \in (0,T)} \ipo \left(1+ |\vd|^2 \right) \d S \leq C ~\text{uniformly with respect to } \delta.
\end{aligned}
\end{equation}
Furthermore, since $\delta$ is not fixed here, we cannot claim that the convective term remains bounded and therefore need more precise estimate on $\vd$. To do so, we recall the interpolation inequality
$$
\|\vd \|_{L^{\frac{5r}{3}}(\o)}^{\frac{5r}{3}} \le C \|\vd \|_{L^{2}(\o)}^{\frac{2r}{3}} \|\vd \|_{V_r}^r.
$$
Then, it follows from the uniform estimate \eqref{adMB7} (recall that $r\ge 6/5$) that
\begin{align}\label{adMB9}
\|\vd\|_{L^{\frac{5r}{3}}(Q)}&\leq C ~\text{uniformly with respect to } \delta.
\end{align}
Next, we explain the definition of $z$ and $z':= \min\{r', 5r/6\}$. Using \eqref{adMB9}, we obtain that
\begin{equation*}
\iT\|(\vd \otimes \vd) \Phi_\delta(|\vd|^2)\|_{L^{\frac{5r}{6}}(\o)}^{\frac{5r}{6}} \d t\leq \iT\|\vd\|^{\frac{5r}{3}}_{L^{\frac{5r}{3}}(\o)}\d t \leq C ~\text{uniformly with respect to } \delta.
\end{equation*}
Then, recalling all above uniform $\delta$-independent estimates, we can also observe the following bound for the time derivative (we skip the computation identical to e.g.~\eqref{uniformptvepsN2})
\begin{align*}
&\iT\|\pt\vd\|^{z'}_{(V_z)^*} \d t\\
&\leq C\iT \left( \|\Sd\|_{L^{r'}(\o)} + \|(\vd \otimes \vd) \Phi_\delta(|\vd|^2)\|_{L^{\frac{5r}{6}}(\o)} +\|\sd\|_{L^2(\po)} + \|\f\|_{(V_r)^*}\right)^{z'}\d t\\
&\leq  C ~\text{uniformly with respect to } \delta.
\end{align*}
Finally, the uniform estimates and the Aubin--Lions lemma, complemented with the Trace theorem conclude that for subsequences,
\begin{subequations}\label{konve}
\begin{align}
\vd &\tow^* \vv &&\text{weakly$^*$ in } L^\infty(0,T;H), \label{vvlih}\\
\vd &\tow \vv &&\text{weakly in } L^r(0,T;V_r),\label{vvvr} \\
(\vd \otimes \vd) \Phi_\delta(|\vd|^2) &\to (\vv \otimes \vv) &&\text{strongly in }~ L^{\rho}(Q) ~\text{ for }~\rho \in \left[1, 5r/6 \right), \label{vdovd}\\
\pt \vd &\tow \pt \vv &&\text{weakly in } L^{z'}(0,T;(V_z)^*), \label{ptvniekde} \\
\vd &\to \vv &&\text{strongly in } L^r(0,T;L^2(\o)), \\
\vd &\to \vv &&\text{strongly in } L^\gamma(Q) \text{ for } \gamma\in \left[1,5r/3 \right), \label{strong53}\\
\Sd &\tow \S &&\text{weakly in } L^{r'}(Q), \label{Sconv}\\
\sd &\tow^* \sg &&\text{weakly$^*$ in } L^\infty(0,T;L^2(\po)), \label{shran}\\
\vd &\tow^* \vv &&\text{weakly$^*$ in } L^{\infty}(0,T; L^2(\po)), \\
\vd &\to \vv &&\text{strongly in } L^1(\Gamma). \label{silnehran}
\end{align}
\end{subequations}
Then, we consider $\vp \in L^z(0,T;V_z)$ in~\eqref{WFNSre}, integrate over $t\in(0,T)$, and after proceeding with $\delta \to 0_+$ while using the results from~\eqref{konve}, we obtain
\begin{equation*}
\iT \langle \pt \vv, \vp \rangle_{V_z} \d t + \iq (\S-(\vv \otimes \vv)): \nabla \vp \d x \d t + \alpha \ig \sg \cdot \vp \d S \d t = \iT \langle \f,  \vp \rangle_{V_z} \d t.
\end{equation*}
Therefore, the weak formulation~\eqref{WFNSr} holds for almost every time $t \in (0,T)$. Moreover, the results \eqref{vvlih}, \eqref{vvvr}, and \eqref{ptvniekde} imply that $\vv \in \C_w([0,T];H)$.

\subsubsection*{Identification on the boundary}

By virtue of \eqref{silnehran}, we can use the Egoroff theorem to get that for every $\zeta>0$ there exists $\Gamma_\zeta$ which satisfies $|\Gamma \setminus \Gamma_\zeta|<\zeta$ and $\vd \to \vv $ strongly in $L^\infty (\Gamma_\zeta)$. Then, using also~\eqref{shran},
\begin{equation*}
\int_{\Gamma_\zeta} \sd \cdot \vd \d x \d t \to \int_{\Gamma_\zeta} \sg \cdot \vv \d x \d t ~~\text{ as }~~ \delta \to 0_+.
\end{equation*}
Then, from Lemma~\ref{Bgraf}, $(\sg,\vv) \in \B$ almost everywhere on $\Gamma_\zeta$, and if we let $\zeta \to 0_+$, we obtain the identification of $\B$ almost everywhere on $\Gamma$, and also that for all $\zeta>0$,
\begin{equation}\label{slabasv1}
\sd \cdot \vd \tow \sg \cdot \vv \text{ weakly in } L^1(\Gamma_\zeta).
\end{equation}

\subsubsection*{Identification inside the domain}

Identification of the graph $\A$ is not so straightforward, especially due to the lack of proper duality pairing in the convective term and consequently in possible non-validity of the energy equality for the limiting equation. We start with subtracting the weak formulation for $\vd$ \eqref{WFNSre} from the one for $\vv$ \eqref{WFNSr}, and integrating the difference over time $(0,T)$, to deduce that
\begin{align*}
\iT \langle \pt (\vd -\vv), \vp \rangle_{V_z} \d t &- \iq \left((\vd \otimes \vd) \Phi_\delta(|\vd|^2)- \vv \otimes \vv\right): \nabla \vp \d x \d t  \\
&+ \iq (\Sd -\S): \DD \vp \d x \d t + \alpha \ig (\sd- \sg) \cdot \vp \d S \d t = 0
\end{align*}
holds for every  $\vp \in L^z(0,T;V_z)$. Consider\footnote{Here the space $\C^\infty_0([0,T];\C^\infty_{0, \diver}(\o))$ is defined as
 $$
 \C^\infty_0([0,T];\C^\infty_{0, \diver}(\o)):=\{\w \in \C^\infty(\overline{Q}); \diver \w = 0 \text{ in } Q, \text{ supp } \w \subset \subset Q\}.
 $$
}
$\vp \in\C^\infty_0([0,T];\C^\infty_{0, \diver}(\o))$, then the boundary term vanishes and we obtain
\begin{align*}
\iq  (\vd -\vv)\cdot \pt \vp \d x \d t = \iq \left((\Sd -\S) + \vv \otimes \vv - (\vd \otimes \vd) \Phi_\delta(|\vd|^2)\right)\!:\! \nabla \vp \d x \d t.
\end{align*}
For further purposes, let us denote
\begin{equation}\label{uGG}
\begin{aligned}
\u^\delta &:= \vd - \vv, \\
\G^\delta_1 &:= \Sd - \S, \\
\G^\delta_2 &:= \vv \otimes \vv - (\vd \otimes \vd) \Phi_\delta(|\vd|^2).
\end{aligned}
\end{equation}
In what follows, we use the result from~\cite[Theorem 2.2 and Corollary 2.4]{BDS}, which we first adapt to our setting.

\begin{lemma}[Breit, Diening, Schwarzacher (2013)]\label{BDS}
Let $Q_0\subset \subset Q$ and let $Q_0 = I_0 \times B_0$. Assume that for $\delta \in (0,1)$,
\begin{align*}
\u^\delta &\tow \0 &&\text{weakly in } L^r(I_0; W^{1,r}_{\diver}(B_0)), \\
\u^\delta &\tow^* \0 &&\text{weakly$^*$ in } L^\infty(I_0; L^2(B_0)), \\
\u^\delta &\to \0 &&\text{strongly in } L^1(Q_0), \\
\G^\delta_1 &\tow \0 &&\text{weakly in } L^{r'}(Q_0), \\
\G^\delta_2 &\to \0 &&\text{strongly in } L^{1+\e}(Q_0).
\end{align*}
as $\delta \to 0_+$. Also, assume that for every $\vp \in \C^\infty_0(I_0;\C^{\infty}_{0, \diver}(B_0))$
\begin{equation*}
\int_{Q_0} \u^\delta \cdot \pt \vp - (\G^\delta_1+\G^\delta_2):\nabla \vp \d x \d t = 0
\end{equation*}
holds, which is a weak formulation of
\begin{equation*}
\pt \u^\delta - \diver (\G^\delta_1+\G^\delta_2) = - \nabla p^\delta.
\end{equation*}
Then there exists $\xi \in \C^{\infty}_0(Q_0)$ such that
\begin{equation}\label{1816}
\chi_{\frac{1}{8}Q_0} \leq \xi \leq \chi_{\frac{1}{6}Q_0},
\end{equation}
and for every $k \in \N$ there exists $\{Q_{\delta,k} \}_{\delta \in (0,1)}$ fulfilling
\begin{equation}\label{qnk}
Q_{\delta,k} \subset Q, ~~\limsup_{\delta \to 0_+} |Q_{\delta,k}| \leq 2^{-k}
\end{equation}
such that for every $\oS \in L^{r'}(Q)$,
\begin{equation}\label{hvi}
\limsup_{\delta \to 0_+} \left| \iq (\G^\delta_1 + \oS)\cdot \nabla \u^\delta \xi \chi_{Q\setminus Q_{\delta,k}} \d x \d t\right|\leq C 2^{\frac{-k}{r}}.
\end{equation}
\end{lemma}

Then the triplet $(\u^\delta, \G^\delta_1, \G^\delta_2)$ defined in \eqref{uGG} satisfies assumptions of the Lemma~\ref{BDS}. Due to Lemma~\ref{pidilema}, for $\D$ we can find $\tS$ such that $(\tS, \D)\in \A$ almost everywhere in $Q$. In \eqref{hvi}, we set $\oS:=\S -\tS$ to get that
\begin{align*}
\limsup_{\delta \to 0_+} &\left| \iq (\Sd - \tS):(\Dd - \D) \xi \chi_{Q\setminus Q_{\delta,k}} \d x \d t \right| \\
&= \limsup_{\delta \to 0_+} \left| \iq (\G^\delta_1 + \oS): \nabla \u^\delta \xi \chi_{Q\setminus Q_{\delta,k}} \d x \d t \right|\leq C 2^{\frac{-k}{r}}.
\end{align*}
Due to \eqref{1816}, $\xi \geq \chi_{\frac{1}{8}Q_0}$ and since $(\tS,\D) \in \A$ and $(\Sd, \Dd) \in \A$, the product in the first integral is non-negative thanks to the monotonicity of $\A$, and we have
\begin{equation}\label{limhvi}
\limsup_{\delta \to 0_+} \int_{\frac{1}{8}Q_0} \left| (\Sd - \tS): (\Dd - \D) \right| \chi_{Q\setminus Q_{\delta,k}}\d x \d t \leq C 2^{\frac{-k}{r}}.
\end{equation}
For any $a\in(0,1)$, we can provide the following computation,
\begin{align*}
\int_{\frac{1}{8}Q_0} &\left| (\Sd - \tS): (\Dd - \D) \right|^a \d x \d t \\
&= \int_{\frac{1}{8}Q_0} \left| (\Sd - \tS): (\Dd - \D) \right|^a \chi_{Q_{\delta,k}}\d x \d t\\
&\quad + \int_{\frac{1}{8}Q_0} \left| (\Sd - \tS): (\Dd - \D) \right|^a \chi_{Q\setminus Q_{\delta,k}}\d x \d t \\
&\leq \left(\int_{\frac{1}{8}Q_0} \left| (\Sd - \tS): (\Dd - \D) \right| \chi_{Q_{\delta,k}}\d x \d t \right)^a |Q_{\delta,k}|^{1-a}\\
&\quad + \left(\int_{\frac{1}{8}Q_0} \left| (\Sd - \tS):(\Dd - \D) \right| \chi_{Q\setminus Q_{\delta,k}}\d x \d t \right)^a |Q|^{1-a}\\
&\leq C |Q_{\delta,k}|^{1-a} + C \left(\int_{\frac{1}{8}Q_0} \left| (\Sd - \tS): (\Dd - \D) \right| \chi_{Q\setminus Q_{\delta,k}}\d x \d t \right)^a.
\end{align*}
Then, as $k \to \infty$, using \eqref{qnk} and \eqref{limhvi}, we obtain that as $\delta \to 0_+$,
\begin{equation*}
\int_{\frac{1}{8}Q_0} \left| (\Sd - \tS): (\Dd - \D) \right|^a \d x \d t \leq C 2^{\frac{-k}{r}} \to 0.
\end{equation*}
However, then also
\begin{equation*}
\left| (\Sd - \tS): (\Dd - \D) \right|^a \to 0 \text{ strongly in } L^1\left(\frac{1}{8}Q_0\right).
\end{equation*}
Due to the Egoroff theorem, for every $\zeta>0$ there exists $Q_\zeta$ such that $|\frac{1}{8}Q_0 \setminus Q_\zeta|\leq \zeta$, and
\begin{equation*}
\left| (\Sd - \tS): (\Dd - \D) \right|^a \to 0 \text{ strongly in } L^\infty(Q_\zeta).
\end{equation*}
Consequently,
\begin{equation}\label{CCC1}
(\Sd - \tS): (\Dd - \D) \to 0 \text{ strongly in } L^\infty(Q_\zeta).
\end{equation}
Since $\lim_{\delta \to 0_+} \int_{Q_\zeta} \tS: (\Dd - \D)\d x \d t = 0 $, which follows from~\eqref{vvvr}, then from~\eqref{CCC1} follows also
\begin{equation*}
\lim_{\delta \to 0_+} \int_{Q_\zeta}\Sd : (\Dd - \D)\d x \d t = 0,
\end{equation*}
which finally implies (using the weak convergence result for $\Sd$ \eqref{Sconv}),
\begin{equation*}
\lim_{\delta \to 0_+} \int_{Q_\zeta}\Sd : \Dd \d x \d t = \int_{Q_\zeta}\S : \D \d x \d t.
\end{equation*}
According to Lemma~\ref{Agraf}, $(\S, \D) \in \A$ almost everywhere in $Q_\zeta$, and we can proceed with~$\zeta \to 0_+$ to obtain the identification of~$\A$ almost everywhere in~$Q$. Also, we have that for all~$\zeta>0$
\begin{equation}\label{slabasd1}
\Sd : \Dd \tow \S : \D \text{ weakly in } L^1(Q_\zeta).
\end{equation}

\subsubsection*{Energy inequality}

For $0<\kappa \ll 1$ and $t \in (0,T-\kappa)$, let $\eta$ be defined as in~\eqref{eta}. We set $\vp:= \vd$ in~\eqref{WFNSre}, multiply it by~$\eta$, and integrate the result over~$\tau \in (0,T)$,
\begin{align*}
\frac{1}{2}\iT \dt \|\vd\|_H^2 \eta \d \tau &- \iq \left((\vd \otimes \vd) \Phi_\delta(|\vd|^2)\right) : \nabla \vd \eta \d x  \d \tau \\
&+ \iq \Sd : \Dd \eta \d x \d \tau + \alpha \ig \sd \cdot \vd \eta \d S \d \tau = \iT \langle \f, \vd \rangle_{V_r} \eta \d \tau.
\end{align*}
Next, we integrate by parts in the first term, use the convective term cancellation due to~\eqref{vorezcanc} and properties of~$\eta$, to obtain
\begin{align*}
\frac{1}{2\kappa} \int_t^{t+\kappa} \|\vd(\tau)\|_H^2 \d \tau &+ \int_{Q_{t+\kappa}} \Sd:\Dd \eta \d x \d \tau + \alpha \int_{\Gamma_{t+\kappa}} \sd \cdot \vd \eta \d S \d \tau \\
&= \int_0^{t+\kappa} \langle \f, \vd \rangle_{V_r} \eta \d \tau + \frac12 \|\vv_0\|_H^2.
\end{align*}
The next step is the limit as $\delta \to 0_+$. For the first term, we can use the weak lower semicontinuity of the $H$-norm. For the products $(\Sd:\Dd)$ and $(\sd \cdot \vd)$, we use the monotonicity of the graphs and that thanks to \eqref{slabasd1} and \eqref{slabasv1}, with the use of the Biting lemma (from~\cite{BaMu}), there exist sequences $\{Q_l\}_{l \in \N}$ and $\{\Gamma_l\}_{l\in\N}$ such that as $l \to \infty$, (for subsequences)
\begin{align*}
|Q \setminus Q_l| &\to 0_+ ~\text{ and }~ \Sd : \Dd \tow \S : \D \text{ weakly in } L^1(Q_l), \\
|\Gamma \setminus \Gamma_l| &\to 0_+ ~\text{ and }~ \sd \cdot \vd \tow \sg \cdot \vv \text{ weakly in } L^1(\Gamma_l).
\end{align*}
For the duality term, we use \eqref{vvvr}, and get that
\begin{align*}
\frac{1}{2\kappa} \int_t^{t+\kappa} \|\vv(\tau)\|_H^2 \d \tau &+ \int_{Q_{t+\kappa}\cap Q_l} \S:\D \eta \d x \d \tau + \alpha \int_{\Gamma_{t+\kappa} \cap \Gamma_l} \sg\cdot\vv \eta \d S \d \tau \\
&\leq \int_0^{t+\kappa} \langle \f, \vv \rangle_{V_r} \eta \d \tau + \frac12 \|\vv_0\|_H^2.
\end{align*}
Next, we proceed with $l \to \infty$, then $Q_{t+\kappa}\cap Q_l \to Q_{t+\kappa}$ and $\Gamma_{t+\kappa} \cap \Gamma_l \to \Gamma_{t+\kappa}$, and finally, thanks to $\vv \in \C_w([0,T];H)$ and the fact that the other terms are well-defined, we can pass with~$\kappa \to 0_+$ to obtain the energy inequality~\eqref{enineqr} for any~$t \in (0,T)$.

\subsubsection*{Initial data attainment}

Similarly as in the previous part, we consider $\eta$ from \eqref{eta}, and multiply \eqref{WFNSre} by this $\eta$, and integrate over $\tau \in(0,T)$,
\begin{align*}
\iT \langle\pt \vd, \vp\rangle_{V_r} \eta \d \tau &- \iq \left((\vd \otimes \vd) \Phi_\delta(|\vd|^2)\right) : \nabla \vp \eta \d x  \d \tau \\
&+ \iq \Sd : \DD \vp \eta\d x \d \tau + \alpha \ig \sd \cdot \vp \eta \d S \d \tau = \iT \langle \f, \vp \rangle_{V_r} \eta \d \tau.
\end{align*}
As $\vp$ is independent of $t$, we can integrate by parts in the first term and subsequently proceed with the limit $\delta \to 0_+$ using the arguments from the previous part and \eqref{vdovd} for the convective term,
\begin{align*}
\frac{1}{\kappa} \int_t^{t+\kappa} (\vv, \vp)_H \d \tau &- \int_{Q_{t+\kappa}} (\vv \otimes \vv) : \nabla \vp \eta \d x  \d \tau + \int_{Q_{t+\kappa}} \S:\DD \vp \eta \d x \d \tau \\
&+ \alpha \int_{\Gamma_{t+\kappa}} \sg\cdot\vp \eta \d S \d \tau = \int_0^{t+\kappa} \langle \f, \vp \rangle_{V_r} \eta \d \tau + \frac12 (\vv_0, \vp)_H^2 \eta(0).
\end{align*}
Due to the arguments that are all explained in the previous sections, we can proceed with~$\kappa \to 0_+$ and~$t \to 0_+$, using that $\vp \in V_z$ is arbitrary and $\vv \in \C_w([0,T];H)$ and obtain
\begin{equation*}
\vv(t) \tow \vv_0 ~\text{ weakly in }~ H.
\end{equation*}
Also, taking the limes superior in the energy inequality \eqref{enineqr}, we obtain that
$$
\limsup_{t \to 0_+} \|\vv(t)\|_H^2 \leq \|\vv_0\|_H^2,
$$
and these two information imply the strong convergence in $H$ as claimed in~\eqref{enineqr}.

\appendix

\section{Orthonormal basis of \texorpdfstring{$V$}{V}}\label{basis}

For $\alpha>0$, define a scalar product on $V$ by
\begin{equation}\label{scalarproduct}
(\u, \vv)_V := \io \DD \u : \D \d x + \alpha \ipo \u \cdot \vv \d S.
\end{equation}
Thanks to the Korn inequality and the definition of the $W^{1,2}$-norm, this scalar product \eqref{scalarproduct} on $V$ is equivalent to the norm on $V$ defined in \eqref{V}. Moreover, one can show, see Lemma~\ref{baza} that there exists a basis of $V$, which is orthogonal in $V$ with respect to the scalar product defined in~\eqref{scalarproduct} and orthonormal in $H$. We denote such basis in what follows as $\{\w_i\}_{i=1}^{\infty}$.

\noindent
\textit{Construction.} Set $V^1 = V$, find $\lambda_1 := \min_{\|\u\|_H=1}(\u,\u)_V$, and denote by $\w_1$ the minimizer, i.e., $\lambda_1 = (\w_1, \w_1)_{V}$.

For every $i \in \N$,
\begin{subequations}\label{baza}
\begin{align}
\text{define } V^{i+1} &:= \{\vv \in V; (\vv, \w_j)_V =0 \text{ for every } j=1, \ldots, i \} \label{Vi}\\
\text{find } \lambda_{i+1} &:= \min_{\u \in V^{i+1}, \|\u\|_H=1}(\u,\u)_V, \label{li}\\
\text{ and denote } &\w_{i+1} \text{ the minimizer, } \lambda_{i+1} = (\w_{i+1}, \w_{i+1})_{V}.\label{wi}
\end{align}
\end{subequations}

\begin{lemma}
The sequence $\{\w_j\}_{j\in \N}$ defined in~\eqref{baza} is a basis of~$V$ and~$H$, it is orthogonal in~$V$ and orthonormal in~$H$. Also, the sequence $\{\lambda_i\}_{i\in \N}$ is non-decreasing with $\lim_{i \to \infty} \lambda_i = +\infty$. For every $i \in \N$, $\lambda_i$  and $\w_i$ solve the problem
\begin{subequations}\label{basisprob}
\begin{align}
- \diver \DD \w_i &= \lambda_i \w_i &&\text{in } \o, \label{ino}\\
\DD\w_i\,\n + \alpha \w_i &= \lambda_i \beta \w_i &&\text{on } \po, \label{onpo}
\end{align}
\end{subequations}
in the weak sense. Moreover, for $P^N$, a projection of $V$ to the linear hull of $\{\w_i\}_{i=1}^N$ defined by
\begin{equation}\label{proj}
P^N\u := \sum_{i=1}^N (\u , \w_i)_H \w_i,
\end{equation}
it holds that for any $\u \in V$
\begin{subequations}\label{vlproj}
\begin{align}
\|P^N \u \|_H &\leq \|\u\|_H, \label{vlH} \\
\|P^N \u \|_V &\leq \|\u\|_V, \label{vlV} \\
P^N \u &\to \u ~\text{ strongly in } V \text{ as } N \to +\infty. \label{vlconv}
\end{align}
\end{subequations}
\end{lemma}
\begin{proof}
Orthogonality in $V$ is evident from the definition of the spaces $V^{i}$ in \eqref{Vi}. The fact that for every $i$, $\|\w_i\|_H=1$, follows from \eqref{li}. We show that for every $i \in \N$, $\w_i$ exists and
\begin{equation}\label{wili}
(\w_i, \vp)_V = \lambda_i (\w_i, \vp)_H ~\text{ for all }~\vp \in V.
\end{equation}
This is a weak formulation of \eqref{basisprob} and also implies orthogonality in $H$.

We start with taking $\{\u^n\}_{n \in \N}$, a minimizing sequence to
\begin{equation*}
(\u^n, \u^n)_V = \io |\DD \u^n|^2 \d x + \alpha \ipo |\u^n|^2 \d S ~\text{ with }~ \|\u^n\|_H = 1 ~\text{ for all }~ n \in \N.
\end{equation*}
From reflexivity of $V$ and its compact embedding in $H$ we get that
\begin{align*}
\u^n &\tow \w_1 &&\text{weakly in } V, \\
\u^n &\to \w_1 &&\text{strongly in } H.
\end{align*}
Therefore $\w_1 \in V$ exists, $\|\w_1\|_H=1$, and for every $\vv \in V$, $\|\vv\|_H=1$, there holds
\begin{equation}\label{mini}
\lambda_1 = (\w_1, \w_1)_V \leq (\vv, \vv)_V
\end{equation}
by weak lower semicontinuity of the norm.

In \eqref{mini}, set $\vv := (\w_1+\e\vp)\|\w_1+\e\vp\|_H^{-1}$ where $\e>0$ and $\vp \in V$ are arbitrary. Then $\vv \in V$, $\|\vv\|_H=1$, and we get that
\begin{align*}
0 &\leq \frac{(\w_1+\e\vp, \w_1+\e\vp)_V }{\|\w_1+\e\vp\|^2_H}- (\w_1, \w_1)_V \\
&= \frac{(\w_1, \w_1)_V }{\|\w_1+\e\vp\|^2_H}- (\w_1, \w_1)_V + 2 \e \frac{(\w_1,\vp)_V }{\|\w_1+\e\vp\|^2_H} + \e^2 \frac{(\vp, \vp)_V }{\|\w_1+\e\vp\|^2_H} \\
&= \frac{\e}{\|\w_1+\e\vp\|^2_H} \left( -\lambda_1(2(\w_1,\vp)_H + \e \|\vp\|^2_H) + 2 (\w_1,\vp)_V + \e (\vp, \vp)_V \right),
\end{align*}
were we used that $(\w_1, \w_1)_V/\lambda_1=1=\|\w_1\|_H$. Next, we divide the expression by $\e$ and take the limit $\e \to 0_+$ to obtain
\begin{equation*}
\lambda_1 (\w_1,\vp)_H \leq (\w_1,\vp)_V .
\end{equation*}
However, it works for $-\vp$ as well, and we obtain the equality
\begin{equation*}
(\w_1,\vp)_V = \lambda_1 (\w_1,\vp)_H ~\text{ for all }~ \vp \in V = V^1.
\end{equation*}
We can do the same for any (fixed) $i \in \N$ to obtain
\begin{equation}\label{phii}
(\w_i,\vp)_V = \lambda_i (\w_i,\vp)_H ~\text{ for all }~ \vp \in V^i.
\end{equation}
The next step is to show that \eqref{phii} is true for any $\vp \in V$, i.e., also for $\vp \in V \setminus V^i$. Note that according to \eqref{Vi},
\begin{equation*}
V = V^1 \supset \ldots \supset V^{i-1} \supset V^i \supset \ldots \implies V \setminus V^{i} \subset \bigcup_{j=1}^{i-1} V^j.
\end{equation*}
Now, let $j<i$ be arbitrary. Due to \eqref{phii}, it holds that
\begin{equation}\label{wjphij}
(\w_j,\vp)_V = \lambda_j (\w_j,\vp)_H ~\text{ for all }~ \vp \in V^j
\end{equation}
and $\w_j \in V^j$. Therefore, set $\vp:=\w_i$ in \eqref{wjphij} (note that $\w_i$ is admissible test function since $\w_i \in V^i \subset V^j$ as $j<i$) to get $(\w_j,\w_i)_V = \lambda_j (\w_j,\w_i)_H$. However, from the definition of $V^i$, $(\w_j,\w_i)_V =0$, and therefore also $(\w_j,\w_i)_H =0$. Since $i \in \N$ and  $j<i$ were arbitrary, we obtain \eqref{wili}.

Next, we study the sequence $\{\lambda_i\}_{i\in \N}$, namely, we want to show that it is non-decreasing with the limit equal to $+\infty$. The first fact is obvious. Regarding the unboundedness, let us assume that it is bounded. Then, from \eqref{wi} and from the reflexivity of $V$, it is weakly convergent in $V$, and from the compact embedding of $V$ into $H$, we get that it converges strongly in $H$, which means that it is Cauchy in $H$. However,
\begin{equation*}
\|\w_i - \w_j\|_H^2 = \|\w_i\|_H^2 + \|\w_j\|_H^2 -2(\w_i, \w_j)_H = 2,
\end{equation*}
which contradicts the Cauchy property.

To show that $\{\w_i\}_{i\in \N}$ is indeed a basis of $V$, we prove two claims: that there are no more eigenvectors $\w_i$, and that there are no more eigenvalues $\lambda_i$.

First, assume that there exists $\vv \in V$ such that $\|\vv\|_V \neq 0$, $\|\vv\|_H = 1$, and $(\vv, \w_i)_V =0$ for every $i \in \N$. The last claim means that $\vv \in V^i$ for every $i$, i.e.,
\begin{equation*}
\lambda_i = \min_{\u \in V^{i}, \|\u\|_H=1}(\u,\u)_V \leq (\vv, \vv)_V.
\end{equation*}
Due to unboundedness of $\{\lambda_i\}_{i \in \N}$, taking the limit $i \to +\infty$ in this inequality results in contradiction with the assumption that $\vv \in V$.

For the second contradiction, assume that there is an eigenvalue $\lambda$, such that $\lambda \neq \lambda_i$ for every $i \in \N$, and that there exists $\w_\lambda \in V$ such that $\|\w_\lambda\|_V \neq 0$, $\|\w_\lambda\|_H = 1$, and
\begin{equation}\label{lam}
(\w_\lambda,\vp)_V = \lambda (\w_\lambda,\vp)_H \text{ for all } \vp \in V.
\end{equation}
For an arbitrary $i \in \N$, use $\vp := \w_\lambda$ in \eqref{wili}, use $\vp:= \w_i$ in \eqref{lam} and subtract from each other to get
\begin{equation*}
(\lambda - \lambda_i) (\w_\lambda,\w_i)_H = 0 \implies (\w_\lambda,\w_i)_H = 0.
\end{equation*}
However, either $(\w_\lambda,\w_i)_H = 0$ for every $i\in\N$ and we are back in the situation from the previous paragraph, i.e., that $\w_\lambda \in V^i$ for every $V^i$, which leads to a contradiction, or there exists $i$ such that $(\w_\lambda,\w_i)_H \neq 0$, but then necessarily $\lambda =\lambda_{i}$ which conflicts the assumption $\lambda \neq \lambda_i$ for every $i \in \N$.
Therefore, $\{\w_i\}_{i\in \N}$ is a basis of $V$, and by density, it is also a basis of $H$.

Finally, we prove the continuity of the projection $P^N$. Note that $\{\w_i/\sqrt{\lambda_i}\}_{i\in \N}$ is orthonormal basis in $V$ and compute
\begin{equation}\label{puu}
\begin{aligned}
(P^N \u, P^N \u)_V &= \left( \sum_{i=1}^N (\u , \w_i)_H \w_i, \sum_{j=1}^N (\u , \w_j)_H \w_j \right)_V
= \sum_{i=1}^N (\u , \w_i)_H^2 (\w_i, \w_i)_V \\
&= \sum_{i=1}^N \lambda_i (\u , \w_i)_H^2
= \sum_{i=1}^N \left(\u , \frac{\w_i}{\sqrt{\lambda_i}}\right)_V^2
\leq \sum_{i=1}^\infty \left(\u , \frac{\w_i}{\sqrt{\lambda_i}}\right)_V^2 = (\u, \u)_V.
\end{aligned}
\end{equation}
In the last equality, we used the fact that (for simplicity, $\vp_i := \frac{\w_i}{\sqrt{\lambda_i}}$ for every $i$)
\begin{equation}\label{basisformula}
\sum_{i=1}^\infty \left(\u , \vp_i \right)_V \vp_i = \u.
\end{equation}
This is true, as it is equivalent to
\begin{equation*}
 \left(\sum_{i=1}^\infty \left(\u , \vp_i \right)_V \vp_i, \vp_j \right)_V = \sum_{i=1}^\infty \left(\u , \vp_i \right)_V (\vp_i,\vp_j)_V = (\u, \vp_j)_V \text{ for every } j\in\N,
\end{equation*}
which holds thanks to the orthonormality of $\{ \vp_j \}_{j \in \N}$ in $V$.

Due to the equivalence of the norm induced by the scalar product on $V$ with the norm on $V$, \eqref{puu} proves the estimate of the $V$-norms \eqref{vlV} and the same arguments are used to estimate the $H$-norms \eqref{vlH} (without the renormalizing by $\sqrt{\lambda_i}$). Also, from the last line of \eqref{puu} it is clear that $\|P^N\u-\u\|_V \to 0$ as $N \to \infty$, i.e., \eqref{vlconv}.
\end{proof}

\bibliographystyle{amsplain}


\begin{thebibliography}{10}

\bibitem{AbFe}
A.~Abbatiello and E.~Feireisl, \emph{On a class of generalized solutions to
  equations describing incompressible viscous fluids}, Ann. Mat. Pura Appl. (4)
  \textbf{199} (2020), no.~3, 1183--1195.

\bibitem{BaMu}
J.~Ball and F.~Murat, \emph{Remarks on {C}hacon’s biting lemma}, Proc. Amer.
  Math. Soc. \textbf{107} (1989), no.~3, 655--663.

\bibitem{BlMaRa}
J.~Blechta, J.~M\'{a}lek, and K.~R. Rajagopal, \emph{On the classification of
  incompressible fluids and a mathematical analysis of the equations that
  govern their motion}, SIAM J. Math. Anal. \textbf{52} (2020), no.~2,
  1232--1289.

\bibitem{BDS}
D.~Breit, L.~Diening, and S.~Schwarzacher, \emph{Solenoidal {L}ipschitz
  truncation for parabolic {PDE}s}, Math. Models Methods Appl. Sci. \textbf{23}
  (2013), no.~14, 2671--2700.

\bibitem{Brezis}
H.~Brezis, \emph{Functional analysis, {S}obolev spaces and partial differential
  equations}, Universitext, Springer, New York, 2011.

\bibitem{BGMS1}
M.~Bul\'{\i}\v{c}ek, P.~Gwiazda, J.~M\'{a}lek, and A.~\'{S}wierczewska
  {G}wiazda, \emph{On steady flows of an incompressible fluids with implicit
  power-law-like rheology}, Adv. Calc. Var. \textbf{2} (2009), no.~2, 109--136.

\bibitem{BGMS2}
\bysame, \emph{On unsteady flows of implicitly constituted incompressible
  fluids}, SIAM J. Math. Anal. \textbf{44} (2012), no.~4, 2756--2801.

\bibitem{BM}
M.~Bul\'{\i}\v{c}ek and J.~M\'{a}lek, \emph{Internal flows of incompressible
  fluids subject to stick--slip boundary conditions}, Vietnam Journal of
  Mathematics \textbf{45} (2017), no.~1, 207--220.

\bibitem{BuMaMa}
M.~Bul\'i\v{c}ek, J.~M\'alek, and E.~Maringov\'a, \emph{On nonlinear problems
  of parabolic type with implicit constitutive equations involving flux},
  arXiv:2009.06917, 2020.

\bibitem{DiRuWo}
L.~Diening, M.~R\r{u}{\v{z}}i{\v{c}}ka, and J.~Wolf, \emph{Existence of weak
  solutions for unsteady motions of generalized {N}ewtonian fluids}, Ann. Sc.
  Norm. Super. Pisa Cl. Sci. \textbf{IX} (2010), no.~1, 1--46.

\bibitem{FrMaSt}
J.~Frehse, J.~M\'{a}lek, and M.~Steinhauer, \emph{On existence results for
  fluids with shear dependent viscosity - unsteady flows}, Partial differential
  equations (Praha, 1998), Chapman \& Hall/CRC Res. Notes Math., vol. 406,
  Chapman \& Hall/CRC, Boca Raton, FL, 2000, pp.~121--129.

\bibitem{H}
S.~G. Hatzikiriakos, \emph{Wall slip of molten polymers}, Prog. Polym. Sci.
  \textbf{37} (2012), 624--643.

\bibitem{Hopf}
E.~Hopf, \emph{\"{U}ber die {A}nfangswertaufgabe f\"{u}r die hydrodynamischen
  {G}rundgleichungen}, Math. Nachr. \textbf{4} (1951), 213--231.

\bibitem{Lady1}
O.~A. Ladyzhenskaya, \emph{New equations for the description of the motions of
  viscous incompressible fluids and global solvability for their boundary value
  problems}, Trudy Mat. Inst. Steklov \textbf{102} (1967), 85--104.

\bibitem{Lady2}
\bysame, \emph{Modifications of the {N}avier--{S}tokes equations for large
  gradients of the velocities}, Zap. Nau\v{c}n. Sem. Leningrad. Otdel. Mat.
  Inst. Steklov (LOMI) \textbf{7} (1968), 126--154.

\bibitem{Lady3}
\bysame, \emph{The mathematical theory of viscous incompressible flow}, Second
  English edition, revised and enlarged. Translated from the Russian by Richard
  A. Silverman and John Chu. Mathematics and its Applications, Vol. 2, Gordon
  and Breach Science Publishers, New York, 1969.

\bibitem{Leray}
J.~Leray, \emph{Sur le mouvement d'un liquide visquex emplissant l'espace},
  Acta Math. \textbf{63} (1934), 193--248.

\bibitem{MaNeRoRu}
J.~M{\'a}lek, J.~Ne{\v{c}}as, M.~Rokyta, and M.~R\accent23u{\v{z}}i{\v{c}}ka,
  \emph{Weak and measure-valued solutions to evolutionary {PDE}s}, Chapman \&
  Hall, London, 1996.

\bibitem{EM}
E.~Maringov\'a, \emph{Mathematical analysis of models arising in continuum
  mechanics with implicitly given rheology and boundary conditions}, Ph.D.
  thesis, Charles University, Prague, 2019.

\bibitem{MaZa}
E.~Maringov\'a and J.~\v{Z}abensk\'y, \emph{On a
  {N}avier-{S}tokes-{F}ourier-like system capturing transitions between viscous
  and inviscid fluid regimes and between no-slip and perfect-slip boundary
  conditions}, Nonlinear Anal. RWA \textbf{41} (2018), 157--178.

\bibitem{Minty}
G.~J. Minty, \emph{{Monotone (nonlinear) operators in Hilbert space.}}, Duke
  Math. J. \textbf{29} (1962), 341--346 (English).

\bibitem{PP}
J.~R.~A. Pearson and C.~J.~S. Petrie, \emph{On melt flow instability of
  extruded polymers}, Polymer Systems: Deformation and Flow (R.~E. Wetton and
  R.~H. Whorlow, eds.), Macmillan, 1968, pp.~163--187.

\bibitem{Raj1}
K.~R. Rajagopal, \emph{On implicit constitutive theories}, Appl. Math.
  \textbf{48} (2003), no.~4, 279--319.

\bibitem{Raj2}
\bysame, \emph{On implicit constitutive theories for fluids}, J. Fluid Mech.
  \textbf{550} (2006), 243--249.

\bibitem{RajSrin}
K.~R. Rajagopal and A.~R. Srinivasa, \emph{On the thermodynamics of fluid
  defined by implicit constitutive relations}, Z. Angew. Math. Phys.
  \textbf{59} (2008), no.~4, 715--729.

\bibitem{Ro}
R.~T. Rockafellar, \emph{On the maximal monotonicity of subdifferential
  mappings}, Pacific J. Math. \textbf{33} (1970), 209--216.

\end{thebibliography}


\providecommand{\bysame}{\leavevmode\hbox to3em{\hrulefill}\thinspace}
\providecommand{\MR}{\relax\ifhmode\unskip\space\fi MR }
\providecommand{\MRhref}[2]{%
  \href{http://www.ams.org/mathscinet-getitem?mr=#1}{#2}
}
\providecommand{\href}[2]{#2}

\end{document}